\documentclass[11pt,a4paper]{amsart}
\usepackage[czech,english]{babel}
\usepackage{amscd,amssymb}
\usepackage[all]{xy}
\usepackage{hyperref}
\title{The $t$-structure induced by an $n$-tilting module}

\thanks{Research supported by Progetto di
Eccellenza Fondazione Cariparo ``Algebraic Structures and their applications''.}

\author[S. Bazzoni]{Silvana Bazzoni}
\address[Silvana Bazzoni]{%
Dipartimento di Matematica \\
Universit\`a di Padova \\
    Via Trieste 63, 35121 Padova (Italy)}
\email{bazzoni@math.unipd.it}

\newcommand{\bbN}{\mathbb{N}}
\newcommand{\bbZ}{\mathbb{Z}}

\newcommand{\Hom}{\operatorname{Hom}}

\newcommand{\End}{\operatorname{End}}

\newcommand{\Ext}{\operatorname{Ext}}
\newcommand{\Tor}{\operatorname{Tor}}
\newcommand{\Cogen}{\operatorname{Cogen}}
\newcommand{\Gen}{\operatorname{Gen}}
\newcommand{\Ker}{\operatorname{Ker}}

\newcommand{\tr}{\operatorname{tr}}
\newcommand{\Img}{\operatorname{Im}}
\newcommand{\Coker}{\operatorname{Coker}}

\newcommand{\cone}{\operatorname{cone}}

\newcommand{\A}{\mathcal{A}}
\newcommand{\B}{\mathcal{B}}
\newcommand{\C}{\mathcal{C}}
\newcommand{\D}{\mathcal{D}}
\newcommand{\E}{\mathcal{E}}
\newcommand{\F}{\mathcal{F}}

\newcommand{\clH}{\mathcal{H}}
\newcommand{\I}{\mathcal{I}}

\newcommand{\K}{\mathcal{K}}

\newcommand{\M}{\mathcal{M}}
\newcommand{\N}{\mathcal{N}}
\newcommand{\clP}{\mathcal{P}}

\newcommand{\clKP}{\mathcal{KP}}
\newcommand{\clKI}{\mathcal{KI}}
\newcommand{\T}{\mathcal{T}}
\newcommand{\U}{\mathcal{U}} 
\newcommand{\V}{\mathcal{V}}

\newcommand{\W}{\mathcal{W}}
 
\newcommand{\Ch}{\mathcal{C}h}   

\newcommand{\Modr}[1]{\mathrm{Mod}\textrm{-}{#1}}

\newcommand{\modr}[1]{\mathrm{mod}\textrm{-}{#1}}

\newcommand{\ModR}{\mathrm{Mod}\textrm{-}R}
\newcommand{\RMod}{R\textrm{-}\mathrm{Mod}}
\newcommand{\modR}{\mathrm{mod}\textrm{-}R}
\newcommand{\Rmod}{R\textrm{-}\mathrm{mod}}

\newcommand{\Add}{\mathrm{Add}}

\newcommand{\Ho}{\operatorname{Ho}}

\newcommand{\Rder}[1]{\mathbf{R}{#1}}
\newcommand{\Lotimes}{\otimes^{\mathbf{L}}}
\newcommand{\RHom}{\Rder\Hom}

\theoremstyle{plain}
\newtheorem{thm}{Theorem}[section]
\newtheorem{lem}[thm]{Lemma}
\newtheorem{prop}[thm]{Proposition}
\newtheorem{cor}[thm]{Corollary}

\newtheorem{fact}[thm]{Fact}

\theoremstyle{definition}
\newtheorem{defn}[thm]{Definition}

\newtheorem{nota}[thm]{Notation}

\theoremstyle{remark}
\newtheorem{rem}[thm]{Remark}

\newtheorem{expl}[thm]{Example}
\newcommand{\etalchar}[1]{$^{#1}$}

\begin{document}

\begin{abstract} We study the $t$-structure induced by an $n$-tilting module $T$ in the derived category $\D(R)$ of a ring $R$. Our main objective is to determine when the heart of the $t$-structure is a Grothendieck category. We obtain characterizations in terms of properties of the module category over the endomorphism ring of $T$ and as a main result we prove that the heart is a Grothendieck category if and only if $T$ is a pure projective $R$-module.\end{abstract}

\maketitle

% -----------------------------------------------------------------------------
% The table of contents
\setcounter{tocdepth}{1}
\tableofcontents

% ------------------------------------------------------------------------------
% The document body
\section*{Introduction}

The notion of $t$-structure in a triangulated category was introduced by Beilinson, Bernstein and Deligne \cite{BBD} in
a geometric context. Its impact and relevance in the algebraic setting has become more and more apparent and 
many constructions of $t$-structures are now available. A first important example is provided by the $t$-structure associated to torsion pairs in abelian categories \cite{HRS}.

 One of the key results about $t$-structures proved in \cite{BBD} is that their heart is an abelian category and a lot of work has been done to determine when the heart of some classes of $t$-structures is a particularly nice category, like a Grothendieck or even a module category.
  For instance it is well know that the heart of the $t$-structure induced by a finitely generated tilting module is equivalent to the module category over the endomorphism ring of the module. Colpi, Gregorio and Mantese \cite{CGM07}  proved that a faithful torsion pair in a module category with torsion free class closed under direct limits, induces a $t$-structure whose heart is a Grothendieck category. In particular this applies to $1$-cotiling torsion pairs.
  {\v{S}}{\v{t}}ov{\'{\i}}{\v{c}}ek \cite{Sto14} generalized this result to an arbitray $n$-cotilting module by using powerful tools from model structures.
  
   A  detailed study of properties of the heart of $t$-structures induced by torsion pairs in a Grothendieck category has been carried on by Parra and Saor{\'{\i}}n \cite{PS}. In particular, they prove that if the torsion class is cogenerating, then the heart is a Grothendieck category 
 if and only if the torsion free class is closed under direct limits.  This applies to the case of a tilting torsion class  and a natural question posed by Saor{\'{\i}}n was to decide if the closure under direct limit of the tilting torsion free class  implies necessarily that the tilting module is finitely generated.
  This has been answered in the paper \cite{BHPST} were it is shown that a tilting torsion free class is closed under direct limits if and only if the tilting module is pure projective and  examples of non finitely generated pure projective tilting modules are exhibited.

  In this paper we consider the $t$-structure induced by an $n$-tilting module $T$ over a ring $R$ and our main interest is to determine when the heart of the $t$-structure is a Grothendieck category. We obtain characterizations in terms of properties of the module category over the endomorphism ring of $T$ and as a main result we prove that the heart is a Grothendieck category if and only if $T$ is a pure projective $R$-module.
  
 The paper is organized as follows. After the necessary preliminaries, in Section~\ref{s:model} we recall the basic definitions and results about model structures and the relations developed by Hoevey between cotorsion pairs in $\Modr R$ and model structures on the category $\Ch (R)$ of unbounded complexes of $R$-modules. In Section~\ref{s:t-structure} we use the model structure on $\Ch (R)$ corresponding to an $n$-tilting cotorsion pair to describe the $t$-structure in the derived category $\D(R)$ of $R$, induced by an $n$-tilting module $T$. In Section~\ref{s:the heart} we study the heart of the $t$-structure, its objects and their cohomologies. In particular, we show that $T$ is a projective generator of the heart and, imitating the arguments on the theory of derivators used by  {\v{S}}{\v{t}}ov{\'{\i}}{\v{c}}ek \cite{Sto14} we show that the inclusion of the heart $\clH$ in $\D(R)$ extends to an equivalence beteween $\D(\clH)$ and $\D(R)$ (Theorem~\ref{T:derivator-equiv}).
  
 In Section~\ref{s:End(T)} we apply the celebrated Gabriel-Popescu's Theorem and a result proved in \cite{BMT} about the derived equivalence induced by a good $n$-tilting $R$-module $T$ between $\D(R)$ and a localization of $\D(S)$, where $S$ is the endomorphism ring of $T$.
We  characterize the case in which the heart is a Grothendieck category, in terms of properties of $S$-modules. In particular, we show that the heart is a Grothendieck category if and only if for every right $S$-module $M$ the derived tensor product $M\Lotimes_S T$ is an object of the heart (Theorem~\ref{T:E-hereditary}).
Moreover, the heart is a Grothendieck category if and only if $S$ admits a two sided idempotent ideal $A$, projective as right $S$-module, such that the canonical morphism $S\to S/A$ is a homological ring epimorphism such that $S/A$ acts as a ``generalized universal localization'' (Theorem~\ref{T:S/A}).

These characterizations allow to describe the direct limits in the heart and in Section~\ref{s:dir-lim-heart}, we show how it is possible to compute direct limits in the heart by means of direct limits of complexes of modules.

Finally in Section~\ref{S:7} we study the consequences of the Grothendieck condition on the heart in terms of closure properties of classes of $R$-modules. This allows us to prove our main result (Theorem~\ref{T:characterization}), that is we show that the heart $\clH$ is a Grothendieck category if and only if the tilting module $T$ is pure projective, generalizing to the case of $n>1$, the result proved in \cite{BHPST}.

We end by studying properties of the trace functor corresponding to an $n$-tilting module and properties of a pure projective $n$-tilting module.
   \section{Preliminaries}~\label{s:preliminaries}
$R$ will be an associative ring with unit. $\ModR$ ($\RMod$) will denote the category of right (left) $R$-modules and $\modR$  ($\Rmod$) the subcategory of finitely presented right (left) $R$-modules.

For more details about the terminology and the results stated in this section we refer to the book \cite{GT12}.

 Given a class 
$\M$  of objects of an abelian category $\C$ and an index $i\geq 0$, we denote by: 
%\[^\perp  \M=\{X\in R{\text{\small -}}\Mod \ | \ \Ext_R^1(X,M)=0 \ {\rm for \ all}\ M\in \M \}, \]
%and by:
\[\M ^{\perp_i}  =\{X\in \C\mid \Ext_{\C}^i(M, X)=0 \ {\rm for \ all}\ M\in \M \}.\]
\[\M ^{\perp}  =\{X\in \C \mid \Ext_{\C}^i(M, X)=0 \ {\rm for \ all}\ M\in \M \ {\rm for \ all\ }  i\in \bbN\}.\]

The classes $^{\perp_i}\M$ and $^\perp \M$ are defined symmetrically.

A pair $(\A, \B)$ of classes of objects of $\C$ is a \emph{%
cotorsion pair} provided that $\mbox{$\A$} = {}^{\perp_1}
\mbox{$\B$}
$ and $\mbox{$\B$} = \mbox{$\A$} ^{\perp_1}$. 

Recall that a full subcategory $\C'$ of an abelian category $\C$ is \emph{resolving } if it is closed under summands, extensions,  kernels of epimorphisms (in $\C'$) and is generating, that is, for every $X\in \A$ there is an epimorphism $C\to X$ with $C\in \C'$.

If moreover the epimorphism $C\to X$ can be chosen functorially in $X$, then $\C'$ is called \emph{functorially resolving}. 

The resolution dimension of an object $X\in \C$ with respect to a resolving subcategory $\C'$ is the minimun integer $n\ge 0$ for which there is an exact sequence $0\to C_n\to \dots\to C_1\to C_0\to X\to 0$ with $C_i\in\C'$ or $\infty $ if such an $n$ doesn't exist.

The notions of coresolving subcategories, functorially coresolving and coresolution dimension are defined dually.
%\end{defn}

Note that for any subcategory  $\C$ of $\Modr R$,
$^\perp \C$ is resolving and in particular, syzygy-closed.  
Dually, $\C^\perp$ is  coresolving and in particular, cosyzygy-closed. 

A cotorsion pair $(\A, \B)$ is called a \emph{hereditary cotorsion pair} if $\A={^\perp \B}$ and $\B=\A ^\perp $. 
%It is easy to see that $(\A, \B)$ is a hereditary cotorsion pair if and only if $(\A,\B)$ is a cotorsion pair and $\A$ is resolving if and only if $(\A,\B)$ is a cotorsion pair and $\B$ is coresolving. 
%

A (hereditary) cotorsion pair $(\A, \B)$ in an abelian category $\C$ is \emph{complete} provided that every object $X\in \C$ admits a \emph{special $\B$-preenvelope}, that is there exists an exact sequence of the form $0\to X \to B \to A\to 0$  with $B\in \B$ and $A\in \A$. Equivalently, every object $X$ admits a \emph{special $\A$-precover}, that is there exists an exact sequence of the form $0\to B \to A \to X\to 0$  with $B\in \B$ and $A\in \A$.

 Preenvelopes and precovers are also called left and right approximations.
For a class $\M$ of objects of an abelian category $\C$, the pair 
$(^{\perp }(\M^{\perp}),\M^{\perp})$ is a (hereditary) cotorsion pair; it is called the cotorsion pair \emph{generated} by $\M$. Symmetrically, the pair $(^{\perp }\M,(^{\perp }\M)^\perp)$ is a (hereditary) cotorsion pair called the cotorsion pair \emph{cogenerated} by $\M$.  Every cotorsion pair generated by a \emph{set } of objects is complete, \cite{Q} or \cite{ET01}. Moreover, every cotorsion pair cogenerated by a class of pure injective objects is generated by a set of objects, hence complete (\cite{ET00}).
% If all the modules in $\C$ have projective dimension $\leq n$, then $^{\perp }(\C^{\perp}) \subseteq \PP_n$ as well.

%If $\mathcal C$ is a set, then a complete description of the modules in ${^{\perp_1} (\mathcal C ^{\perp_1})}$ is available. In fact, by results in  \cite{ET} and by~\cite[Theorem 22]{T}, a module belongs to ${^{\perp_1} (\mathcal C ^{\perp_1})}$ if and only if it is a direct summand of a $\mathcal C'$-filtered module where $\mathcal C' = \mathcal C \cup \{R\}$. Clearly,
%${^{\perp} (\mathcal C ^{\perp})} = {^{\perp_1} (\mathcal C ^{\perp_1})}$ provided that a first syzygy of $M$ is contained in $\mathcal C$ whenever $M \in \mathcal C$.

%
For every $R$-module $M$, $\Add M$ will denote the class of modules isomorphic to
summands of direct  direct sums of copies of $M$, 
and $\Gen M$ will denote the class of all epimorphic images of direct sums of 
copies of $M$. 

\begin{defn}\label{D:n-tilting} A right $R$-module $T$ is $n$-tilting if it satisfies the following conditions:
\begin{enumerate}
\item[(T1)] p.d.$T\leq n$;
\item[(T2)] $\Ext_R^i(T, T^{(\lambda)})=0$ for every cardinal $\lambda$ and every $i\geq 1$;
\item[(T3)] there exists an $r\geq 0$ and an exact sequence:
\[0\to R\to T_0\to T_{1}\to\dots \to T_r\to 0,\] where  $T_i\in \Add T$, for every $0\leq i\leq r$.
\end{enumerate}
A finitely generated $n$-tilting module is called {\sl classic}.
%An $R$-module $T$ satisfying conditions (T1) and (T2) is called {\sl partial $n$-tilting}.
\end{defn}

If $T$ is a $n$-tilting module, $T^\perp$ is called $n$-tilting class and the cotorsion pair $(\A, T^\perp)$ generated by $T$ is called $n$-tilting cotorsion pair. The kernel $\A\cap T^\perp$ of the cotorsion pair coincides with $\Add T$.
%A $1$-tilting module gives rise to a torsion pair $(\T, \F)$ where $\T=T^\perp$ is called  a tilting class  and $\F=\{M\in \ModR\mid \Hom_R(T, M)=0\}.$ 
Two $n$-tilting modules $T$ and $U$ are said to be {\sl equivalent} if $T^\perp =U^\perp$, or equivalently if $\Add T=\Add U$.

By \cite{BH08, BS07} an $n$-tilting class is of \emph{finite type}, that is there is a set $ \mathcal S$ of modules in $\A$ with a projective resolution consisting of finitely generated projective modules such that $\mathcal S^\perp= T^\perp$. In Crawley-Boevey terminology (see \cite{CB1}) this means that, for every $S\in \mathcal S$ the functors $\Ext^i_R(S, -)$ are coherent, hence that $n$-tilting classes are \emph{definable}, that is  they are closed under direct products, direct limits and pure submodules.

Note that all the syzygies of a classical $n$-tilting module are finitely generated (see \cite[Corollary 3.9]{BH09}).
%We will also consider the complete cotorsion pair $(\A, T^\perp)$ induced by a partial $n$-tilting module

Recall that a module is pure projective
 if and only if it has the projective property with respect to pure exact sequences.

By Warfield~\cite{War1} a module $M$ is pure projective if and only if every pure exact sequence $0\to A \to B \to M\to 0$ splits or, equivalently, if and only if it is a direct summand of a direct sum of finitely presented modules.

%We recall the notion of pure projectivity.
%\begin{defn} An $R$-module is pure projective if and only if it has the projective property with respect to the pure exact sequences.
%\end{defn}
%
%By Warfield~\cite{War1} the following are equivalent formulations:
%\begin{enumerate}
%\item[-] A module is pure projective if and only if it is a direct summand of a direct sum of finitely presented modules.
%\item[-] A module $M$ is pure projective if and only if every pure exact sequence $0\to A \to B \to M\to 0$ splits.
%\end{enumerate}

\section{Model structures}~\label{s:model}

We describe some \emph{model structures} on the category $\Ch(R)$ of unbounded complexes of $R$-modules whose homotopy category is the derived category $\D (R)$ of $R$.

For the definition of a model structure we refer to the book by Hoevy \cite{Hov99} or to the survey \cite{Sto13}. 

We just recall that a model structure on a category $\C$ consists
of three classes of morphisms Cof, W, Fib called \emph{cofibrations}, \emph{weak equivalences} and \emph{fibrations}, respectively, satisfying certain axioms.

A \emph{model category} $\C$ is an abelian cocomplete category with a model structure.
 An object $X$ in a pointed model category $ \C$ is called \emph{cofibrant (trivial)} if
the unique morphism from the initial object to $X$ is a cofibration (a weak equivalence) and it is called \emph{fibrant} if the unique morphism from $X$ to the terminal object is a fibration. 
 
 The homotopy category $\Ho \C$ is obtained by  formally inverting all morphisms in W.

An \emph{abelian model structure} on an abelian category $\C$ is a model structure such that cofibrations (fibrations) are the monomorphisms (epimorphisms) with cofibrant (fibrant) cokernels (kernels).
% and an abelian model category is a model category with an abelian model structure.

We recall a method discovered by Hoevey~ \cite{Hov02}, \cite{Hov07} and developed by Gillespies~\cite{G2, G3} and other authors    which allows to define a model structures on the category $\Ch(\C)$ of unbounded complexes over $\C$ starting from a complete cotorsion pair on $\C$. We state Hoevey's result only in the situation needed in the sequel. 

If $\C$ is a subclass of $\Modr R$ closed under extensions, following the notations used by Gillespie (see \cite{G2, G3} or \cite{Sto13}), we denote by $\tilde{\C}$ the class of all acyclic complexes of $\Ch (R)$ with terms  in $\C$  and cocycles in $\C.$

\begin{prop}\label{P:Hoevey}(\cite{Hov07}, \cite{G2, G3}) If $(\A, \B)$ is a  cotorsion pair in $\Modr R$ generated by a set of modules, there is an abelian model structure on $\Ch (R)$ given as follows:
\begin{enumerate}
\item Weak equivalences are quasi-isomorphisms.
\item Cofibrations (trivial cofibrations) are the monomorphism $f$ such that $\Ext^1_{\Ch (R)}(\Coker f, X)=0$, for every $X\in \tilde{\B}$ ($\Coker f\in \tilde{\A}$) and $C$ is a cofibrant object if and only if $\Ext_{\Ch (R)}(C, X)=0$, for every $X\in \tilde{\B}$.
\item Fibrations (trivial fibrations) are the epimorphisms $g$ such that\\ $\Ext^1_{\Ch (R)}(X, \Ker g)=0$, for every $X\in \tilde{\A}$ ($\Ker g\in \tilde{\B}$) and $F$ is a fibrant object if and only if $\Ext_{\Ch(R)}(X, F)=0$, for every $X\in \tilde{\A}$.
\end{enumerate}
The homotopy category of this model structure is the derived category $\D(R)$ of $R$.

Moreover, if $\C$, $\W$, $\F$ are the classes of cofibrant, trivial (acyclic) and fibrant objects, respectively, then $(\C, \W\cap \F)$ and $(\C\cap \W, \F)$ are complete cotorsion pairs in $\Ch (R)$.
\end{prop}
This allows to describe the morphisms in $\D(R)$. In fact, if $X$ is a cofibrant object in $\Ch (R)$ and $Y$ is a fibrant object in $\Ch (R)$, then:
\[\Hom_{\K(R)}(X, Y)=\Hom_{\D(R)}(X, Y),\]
where $\K(R)$ is the homotopy category of $R$.

To describe the cofibrant and fibrant objects in the model structure induced by a complete cotorsion pair we will make use of the following well known formula:
\[(\ast)\quad \Ext^1_{dw}(X[1], Y)\cong \Hom_{\K(R)}(X, Y),\]
where $\Ext^1_{dw}$ denotes the subgroup of $\Ext^1_{\Ch (R)}$ consisting of the degreewise splitting short exact sequences.

\begin{expl}\label{E:1} If $\clP$ is the class of projective $R$-modules, the model structure induced by the cotorsion pair $(\clP, \Modr R)$ is called the canonical projective model structure: The trivial objects are the acyclic complexes, the cofibrant objects (also called $K$-projective) are the complexes $X$ with projective terms such that $ \Hom_{\K(R)}(X, N)=0$ for every acyclic complex $N$ and every complex is a fibrant object. Moreover, if $\clKP$ is the class of $K$-projective complexes and $\N$ the class of acyclic complexes, the pair $(\clKP, \N)$ is a complete cotorsion pair in $\Ch (R)$.
 \end{expl}
\begin{expl}\label{E:2} Symmetrically, if $\I$ is the class of injective $R$-modules, the model structure induced by the cotorsion pair $( \Modr R, \I)$ is called the canonical injective model structure: The trivial objects are the acyclic complexes, the fibrant objects (also called $K$-injective) are the complexes $Y$ with injective terms such that $ \Hom_{\K(R)}( N, Y)=0$ for every acyclic complex $N$ and every complex is a cofibrant object. Moreover, if $\clKI$ is the class of $K$-injective complexes and $\N$ the class of acyclic complexes, the pair $( \N, \clKI)$ is a complete cotorsion pair in $\Ch (R)$.
 \end{expl}

\begin{rem}\label{R:bounded-complex} If $(\A, \B)$ is a hereditary cotorsion pair in $\Modr R$ and $X$ is a bounded above complex with terms in $\A$, then $X$ is cofibrant in the model structure induced by $(\A, \B)$ (the proof is similar to the proof that a bounded above complex with projective terms is $K$-projective  (see \cite[Lemma 2.3.6]{Hov99}).
Dually, if $X$ is a bounded below complex with terms in $\B$, then $X$ is fibrant.
\end{rem}
We are now in a position to describe some particular model structures: the model structure induced by a module of finite homological dimension. The following results  is obtained by imitating the arguments used by \v{S}{\v{t}}ov{\'{\i}}{\v{c}}ek~\cite[Theorem 3.17]{Sto14} to describe the model structure induced by a cotilting module.

\begin{thm}\label{T:M-model-structure} Let $M$ be an $R$-module with p.d. $M\leq n$. Let $(\A, \B)$ be the hereditary cotorsion pair generated by $M$. There is an abelian model structure on $\Ch (R)$ described as follows:
\begin{enumerate}
\item Cofibrations (trivial cofibrations) are the monomorphism $f$ such that $\Ext^1_{\Ch (R)}(\Coker f, X)=0$, for every $X\in \tilde{\B}$ ($\Coker f\in \tilde{\A}$).
\item Fibrations (trivial fibrations) are the epimorphisms $g$ such that $\Ker g$ has terms in $ \B$ ($\Ker g \in \tilde{\B}$).
\end{enumerate}
\end{thm}

\begin{proof} In view of Proposition~\ref{P:Hoevey} the only thing which remains to be proved is the description of the fibrant objects. The statement  will follow by the next lemma.\end{proof}
\begin{lem}\label{L:fibrant} Let $M$ be as in Theorem~\ref{T:M-model-structure} and let $X\in \Ch (R)$. Then, $Y$ is a fibrant object in the model structure induced by the hereditary cotorsion pair $(\A, \B)$ generated by $M$ if and only if $Y$ has all the terms in $\B$.
\end{lem}
\begin{proof} The proof is inspired by  \cite[Theorem 3.17]{Sto14}, but we give the details for the sake of completeness.

Let $Y$ be a fibrant object. Then $\Ext^1_{\Ch (R)}(X, Y)=0$, for all $X\in \tilde{\A}$. For every $A\in \A$, let $D^n(A)$ be the complex defined by $0\to A\overset{1_A}\to A \to 0$ with $A$ in degrees $n$ and $n+1$. Then $D^n(A)\in \tilde{\A}$ and by 
%well known formulas $\Ext^1_{dw}(D^n(A), Y)\cong \Hom_{\K(R)}(D^n(A)[-1]Y)$ (see e.g
 \cite[Lemma 3.1.5]{G3})  $\Ext^1_{\Ch (R)}(D^n(A), Y)\cong \Ext^1_R(A, Y^n)$, hence $Y^n\in \B$.
 
 Conversely, assume that all the terms $Y^i$ of a complex $Y$ are in $\B$. We claim that $Y$ is fibrant that is $\Ext^1_{\Ch (R)}(X, Y)=0$ for every $X\in \tilde{\A}$. Clearly $\Ext^1_{\Ch (R)}(X, Y)\cong \Ext^1_{dw}(X, Y)$ and by formula $(\ast)$, $\Ext^1_{dw}(X[1], Y)\cong \Hom_{\K(R)}(X, Y)$. 
 
  Consider a special preenvelope 
 \[0\to Y\to I_0\to N\to 0,\] of $Y$ with respect to the complete cotorsion pair $(\N, \clKI)$ in $\Ch(R)$ (described in Example~\ref{E:2}).
 
 Then $N$ is an exact complex and, since $\B$ contains the injective modules, all the terms $I_0^i$ of $I_0$ are in $\B$, hence also the terms $N^i$ are in $\B$, since $\B$ is coresolving.
  For every $i\in \bbZ$ consider the short exact sequences
  $0\to \Ker d^i_N\to N^i\to \Ker d^{i+1}_N\to 0$. By dimension shifting and by the condition p.d.$M\leq n$ we conclude that $\Ext^j_R(M, \Ker d^i_N)=0$, for every $i, j\in \bbZ$. So $N\in \tilde{\B}$.  Thus, for every $X\in \tilde{\A}$, \[(a)\quad \Ext^1_{\Ch (R)}(X, N) \cong \Ext^1_{dw}(X, N)\cong\Hom_{\K(R)}(X[-1], N))=0.\] Consider the triangle
  \[N[-1]\to Y \to I_0\to N\] and let $X\in \tilde{\A}$. We have an exact sequence
  \[\Hom_{\K(R)}(X, N[-1])\to \Hom_{\K(R)}(X, Y)\to  \Hom_{\K(R)}(X, I_0),
  \]
  where the last term vanishes since $I_0$ is $K$-injective (and $X$ is acyclic) and the first term vanishes by  $(a)$. So $Y$ is fibrant. \end{proof}
%%that $\B=T^{\perp}$.  In this case, $({}
\begin{cor}\label{C:hom-in-homotopy} In the notations of Theorem~\ref{T:M-model-structure}, let $Y$ be a complex with terms in $\B$ and let $Z$ be cofibrant in the model structure induced by $M$. Then, there is a natural isomorphism
\[\Hom_{\K(R)}(Z, Y)\cong \Hom_{\D(R)}(Z, Y).\]

In particular this applies to the complexes $Z$ bounded above and with terms in $\A$.

\end{cor}

If $C$ is a pure injective module, then by Auslander's result every cosyzygy of $C$ is pure injectve and thus, by \cite{ET00} the  hereditary cotorsion pair $(^\perp C, (^\perp C,)^\perp)$ cogenerated by $C$ is complete.

This observation allows to state also a dual of Theorem~\ref{T:M-model-structure}.
\begin{thm}\label{T:C-model-structure} Let $C$ be a pure injective $R$-module with i.d. $C\leq n$. Let $(\A, \B)$ be the hereditary cotorsion pair cogenerated by $C$. There is an abelian model structure on $\Ch (R)$ described as follows:
\begin{enumerate}
\item Cofibrations (trivial cofibrations) are the monomorphism $f$ such that $\Coker f$ has terms in $\A$ ($\Coker f\in \tilde{\A}$).
\item Fibrations (trivial vibrations) are the epimorphisms $g$ such that\\ $\Ext^1_{\Ch (R)}(X, \Ker g)=0$, for every $X\in \tilde{\A}$ ($\Ker g\in \tilde{\B}$).
\end{enumerate}
\end{thm}
\begin{proof} It is enough to dualize the proof of Theorem~\ref{T:M-model-structure}. To prove that every complex $Z $ with terms in $\A$ is cofibrant, one uses special precovers with respect to the complete cotorsion pair  $(\clKP, \N)$ in $\Ch(R)$ (described in Example~\ref{E:1}).
\end{proof}

\begin{cor}\label{C:hom-in-homotopy-dual} In the notations of Theorem~\ref{T:C-model-structure}, let $Z$ be a complex with terms in $\A$ and let $Y$ be fibrant in the model structure induced by $C$. Then, there is a natural isomorphism
\[\Hom_{\K(R)}(Z, Y)\cong \Hom_{\D(R)}(Z, Y).\]

In particular this applies to the complexes $Y$ bounded below and with terms in $\B$.

\end{cor}

\section{$t$-structures} ~\label{s:t-structure}
 \begin{defn}\label{d:BBD}(Beilinson, Bernstein, Deligne '82) A $t$-structure in a triangulated category $(\D, [-])$ is a pair
$(\U, \V)$ of subcategories such that: 
\begin{enumerate}
\item
$\U[1]\subseteq \U$;

\item $\V=\U^{\perp}[1]$ where $\U^{\perp}=\{Y\in \D\mid \Hom_{\D}(\U, Y)=0\}$;
\item
  for every object $D\in \D$ there is a triangle $U\to D\to Y\to$ with $U\in \U$ and $Y\in  \U^{\perp}$.
\end{enumerate}
\end{defn}
\begin{thm}\label{T:BBD} (\cite{BBD}) The heart $\clH=\U\cap \V$ of a $t$-structure $(\U, \V)$ is an abelian category.
\end{thm}

\begin{nota}\label{N:notation}
Let $T$ be an $n$-tilting mode. We denote by $\U$ and $\V$ the following full subcategories of $\D(R)$:
\begin{enumerate}
\item $\U=\{X\in \D(R)\mid \Hom_{\D(R)}(T[i], X)=0,\textrm{for\ all\ } i<0\}$.
\item $\V=\{Y\in \D(R)\mid \Hom_{\D(R)}(T[i], Y)=0, \textrm{for\ all\ }i>0\}$
%\item The heart $\clH$ consists of the complexes $\{Z\in \D(R)\mid \Hom_{\D(R)}(T[i], Y)=0,\forall i\neq 0\}.$
\end{enumerate}
\end{nota}
Our aim is to prove that the pair $(\U, \V)$ in Notation~\ref{N:notation} is a $t$-structure in $\D(R)$.
 Following the pattern of the proof of \cite[Lemma 4.4]{Sto14} we can give a description of the complexes in $\U$.
\begin{lem}\label{L:U} Let $T$ be an $n$-tilting module, $T^\perp$ the corresponding tilting class and $\U$ as in Notation~\ref{N:notation}. For a complex $X\in \D(R)$ the following are equivalent:
\begin{enumerate}
\item
 $X\in \U$.
\item $X$ is isomorphic in $\D(R)$ to a complex of the from
 \[\dots \to X^{-n}\to X^{-n+1}\to\dots\to X^{-1}\to X^0\to 0\to 0\dots,\]
 with $X^{-i}\in T^\perp$ for every $i\geq 0$.
\item  $X$ is isomorphic in $\D(R)$ to a complex as in (2) with   $X^{-i}\in \Add T$, for every $i\geq 0$.

  \end{enumerate}
  \end{lem}
 \begin{proof}  (3) $\Rightarrow $ (2) is obvious. (2) $\Rightarrow $ (1) follows by Corollary~\ref{C:hom-in-homotopy}, since it is obvious that $\Hom_{\K(R)}(T[i], X)=0$ for every $i<0$.
  
  It remains to show that (1) $\Rightarrow $ (3) Let $X\in \D(R)$. By Theorem~\ref{T:M-model-structure}, we can assume that $X$ has terms in $T^\perp$. By induction we construct a complex 
   \[Z=\dots \to Z^{-n}\to Z^{-n+1}\to\dots\to Z^{-1}\to Z^0\to 0\to 0\dots,
   \] 
 with terms $Z^i\in \Add T$ and a cochain map $f\colon Z\to X$ which becomes an isomorphism in $\D(R)$.
 
 Let $d^{-i}\colon X^{-i}\to X^{-i+1}$ be the $i^{th}$-differential of $X$. Consider an $\Add T$-precover of $\Ker d^0$, like for instance the canonical morphism \[Z^0=T^{(\Hom_R(T, \Ker d^0))}\overset{\phi}\to \Ker d^0\] and let $f^0$ be the composition of $\phi$ with the inclusion $\Ker d^0\to X^0$. By induction construct $f^{-i-1}\colon Z^{-i-1}\to X^{-i-1}$ in the following way. Having defined $f^{-i}$ and $\delta^{-i}\colon Z^{-i}\to Z^{-i+1}$,  let $K^{-i}$ be the kernel of $\delta^{-i}$ and let $g^{-i}$ be the  composition $K^{-i}\to Z^{-i}\overset{f^{-i}}\to X^{-i}$.
 Consider the pullback $P^{-i-1}$ of  the maps $g^{-i}$ and $d^{-i-1}$ and let $Z^{-i-1}\to P^{-i-1}$ be an $\Add T$-precover of $P^{-i-1}$. Then let $f^{-i-1}\colon Z^{-i-1}\to X^{-i-1}$ be the obvious composition. Visually we have:
 
  \[\xymatrix{Z^{-i-1}\ar@/_2pc/[dd]_{f^{-i-1}}\ar[d]\ar[rr]^{\delta^{-i-1}}&&Z^{-i}\ar[dd]^{f^{-i}}\\
P^{-i-1}\ar[d]\ar[r]&K^{-i}\ar@{^(->}[ur]\ar[dr]^{g^{-i}}\\
X^{-i-1}\ar[rr]_{d^{-i-1}}&&X^{-i}}
\] 
We claim that the cochain map $f=(f^{-i})_{i}$ is an isomorphism in $\D(R)$. By Corollary~\ref{C:hom-in-homotopy}, $f$ induces a morphism 
\[\Hom_{\K(R)}(T[i], Z)\overset{\Hom_{\K(R)}(T[i], f)}\longrightarrow\Hom_{\K(R)}(T[i], X),\]
with $\Hom_{\K(R)}(T[i], f)=0$ for every $i<0$ by the assumption on $X$ and by the fact that $Z^j=0$ for every $j>0$. If $i\geq 0$, then $\Hom_{\K(R)}(T[i], f)=0$ by construction, since $Z^{-j}$ are $\Add T$-precovers, for every $j\geq 0$.

From the mapping cone: $Z\overset{f}\to X\to \cone f\to Z[1]$ we obtain \[\Hom_{\K(R)}(T[i], \cone f)\cong\Hom_{\D(R)}(T[i], \cone f)=0,\] for all $i\in \bbZ$, since $\cone f$ is fibrant. Thus we conclude that $\cone f=0$, since $T$ is a generator of $\D(R)$ which yields that $f$ is a quasi-isomorphism.
\end{proof}

\begin{thm}\label{T:n-tilting-t-structure} The pair $(\U, V)$ defined in Notation~\ref{N:notation} is a $t$-structure called the $t$-structure induced by $T$ and its heart $\clH$ is given by

\[\clH=\{Y\in \D(R)\mid \Hom_{\D(R)}(T[i], Y)=0 \textrm {\ for\ all\ } i\neq 0\}.\]

\end{thm}
\begin{proof} From the description of the objects in the subcategory $\U$ it follows that $\U$ is a {\sl pre-aisle}, that is, if $X\in \U$ then also $X[1]\in \U$ and for every triangle $X \to Y\to Z\to X[1]$ in $\D(R)$, if $X, Z\in \U$ also $Y\in U$.

Then $\U$ is the smallest cocomplete subcategory of $\D(R)$ containing $T$. By \cite[Lemma 3.1, Proposition 3.2]{AJS2} $(\U, \U^\perp[1])$ is a $t$-structure and $\U^\perp[1]=\V$. The description of the heart is now obvious.

\end{proof}

\begin{rem}\label{R:n=1} If $T$ is a $1$-tilting module, and $\F=\{N_R\mid \Hom_R(T, N)=0\}$, then $(T^\perp, \F)$ is a torsion pair in $\Modr R$ and the $t$-structure induced by $T$ defined in Theorem~\ref{T:n-tilting-t-structure}  coincides with the $t$-structure induced by the torsion pair $(T^\perp, \F)$ as defined in \cite{HRS}, that is \[\U=\{X\in \D(R)\mid H^0(X)\in T^\perp \textrm{\ and\ } H^i(X)=0\,  
\textrm{\ for\ all\ } i >0\};\] \[\V=\{Y\in \D(R)\mid H^{-1}(X)\in \F \textrm{\ and\ }H^i(Y)=0 \textrm{\ for\ all\ } i<-1,\] so that the objects of  $\clH$ are isomorphic to complexes of the form $0\to X^{-1}\overset{d^{-1}}\to X^0\to 0$  with $\Ker d^{-1}\in \F$ and $\Coker d^{-1}\in T^\perp$. 
\end{rem}%
\section{The heart of the $t$-structure induced by an $n$-tilting module}~\label{s:the heart}

In this section $\clH$ will always denote the heart of the $t$-structure induced by an $n$-tilting module $T$.%

It is well known that $\clH$ satisfies the following properties
\begin{enumerate}
\item If $X, Z\in \clH$ and $X\to Y\to Z\to X[1]$ is a triangle in $\D(R)$, then $Y\in \clH$
\item A sequence $0\to X\to Y\to Z\to 0$ is exact in $\clH$ if and only if $X\to Y\to Z\to X[1]$ is a triangle in $\D(R)$.
\item For every $X, Y\in \clH$, $\Ext^1_{\clH}(X, Y)\cong \Hom_{\D(R)}(X, Y[1])$.
\end{enumerate}

\begin{rem}\label{R:Saorin}(\cite[Lemma 3.1]{PS15}) The inclusion functor $\iota\colon \clH \to \U$ admits a left adjoint which can be defined using the cohomological functor $\tilde{H}\colon\D(R)\to \clH$ constructed in \cite{BBD}. Given $X\in\U$, let \[U\to X[-1]\to Z\to U[1]\] be a triangle in $\D(R)$ with $U\in \U$ and $Z\in \U^\perp$. Then, a left adjoint $b\colon \U\to \clH$ is defined by letting $b(X)=Z[1]$.
\end{rem}

When applicable, we will make use of Corollary~\ref{C:hom-in-homotopy} without explicitly mentioning it.

First of all we note the following
\begin{prop}\label{P:-n-0-homology} Let $T$ be an $n$-tilting module. Then $T$ is a projective object of the heart $\clH$ and for every complex $X$ in $\clH$, $H^i(X)=0$ for every $i>0$ and $i<-n$.
\end{prop}
\begin{proof} By property (3) above $\Ext^1_{\clH}(T, X)\cong \Hom_{\D(R)}(T, X[1])$, so $\Ext^1_{\clH}(T, X)=0$, by the description of the objects in $\clH$. Hence, $T$ is a projective object of $\clH$. By \cite[Proposition 3.5]{B04} we can chose a sequence $E$ as in (T3) of Definition~\ref{D:n-tilting} with $r=n$. Apply the functor$ \Hom_{\D(R)}(T[i], -)$ to the triangles in $\D(R)$ corresponding to the short exact sequences in which the exact sequence $E$ splits and consider  the long exact sequences in cohomology associated to the short exact sequences. For every $X\in \clH$ we get  $\Hom_{\D(R)}(R[i], X)\cong \Hom_{\K(R)}(R[i], X)=0$ for every $i\neq 0, 1, 2,\dots, n$. Hence $X$ has cohomology only in degrees $0, -1, -2, \dots, -n$.
\end{proof}

Let \[\clH_i=\{X\in \clH\mid H^{-j}(X)=0 \textrm {\ for\ every\ } j>i\}.\]
 Thus, $\clH_0=T^\perp[0]$, $\clH_i\subseteq \clH_{i+1}$ and by Proposition~\ref{P:-n-0-homology}, $\clH= \clH_n$.
 
 The next result is obtained by dualizing the proofs of \cite[Lemma 5.18, Proposition 5.20]{Sto14}.
\begin{prop}\label{P:resol-dimension}  For every $X\in \clH_i$ there is an exact sequence
\[0\to Y\to T_0[0]\to X\to 0\] in $\clH$ with $T_0\in \Add T$ and $Y\in \clH_{i-1}$. In  particular, $T$ is a projective generator of $\clH$, $\Add T$ is equivalent to the full subcategory of projective objects of $\clH$  and the $T^\perp$-resolution dimension of an object  in $ \clH$ is at most $n$.
\end{prop}

\begin{proof}
Let $X\in \clH_i$. By Lemma~\ref{L:U} we may assume that  \[X= \dots \to X^{-n}\to X^{-n+1}\to\dots\to X^{-1}\to X^0\to 0,\] with $X^i\in \Add T$. Consider the complex $X^0[0]$ and the obvious chain map  $f\colon X
^0[0]\to X$. We have a triangle in $\D(R)$
\[ X^0[0]\to X\to Z\to ,\]

where $Z$ is fibrant and we may assume that \[Z: \dots\to 0\to X^{-n}\overset{d^{-n}}\to X^{-n+1}\to \dots \to X^{-1} \to 0\to 0\]
with $X^{-i}$ in degrees $-i$. Applying the cohomological functor $\Hom_{D(R)}(T, -)$ to the triangle and using condition $(T2)$ of tilting modules, we obtain that  $\Hom_{D(R)}(T[i], Z)=0$, for every $i\neq 0, 1$. Moreover, $\Hom_{D(R)}(T[1], Z)\cong \Hom_{K(R)}(T[i], Z)$  (by Corollary~\ref{C:hom-in-homotopy}) and by the choice of $Z$ we have $\Hom_{K(R)}(T[0], Z)=0$. Thus $Z[-1]\in \clH$ and computingg the homologies of the terms in the triangle, we see that $Z[-1]\in \clH_{i-1}$. Thus, the triangle $Z[-1]\to X^0[0]\to X\to$ gives the wanted exact sequence $0\to Z[-1]\to X^0[0]\to X\to 0$ in $\clH$.
 The last statements are now obvious. \end{proof}

\begin{prop}\label{P:co-resolving} $\Add T$ and $T^\perp$ are functorially resolving subcategories of $\clH$.
\end{prop}
\begin{proof} $\Add T$ and $T^\perp$ are closed under summands and extensions and by Proposition~\ref{P:resol-dimension} they are generating subcategories in $\clH$.
We need to prove their closure under kernels of epimorphisms. Let $0\to X\to Y \to Z\to 0$ be an exact sequence in $\clH$ with $Y,Z\in T^\perp$ and consider the triangle $X\to Y\overset{f}\to Z\to X[1]$ in $\D(R)$. Then $X$ is quasi isomorphic to 
\[\dots 0\to0\to Y\to Z\to0\to o\dots\]
with $Y$ in degree $0$ and $Z$ in degree $1$, hence, by Lemma~\ref{L:fibrant} $X$ is a fibrant object. 
Let $0\to B_Z\to A_Z\overset{\pi}\to Z\to 0$ be a special $\A$-precover of $Z$ in the cotorsion pair $(\A, T^\perp)$ in $\Modr R$. Then $A_Z\in \A\cap T^\perp=\Add T$. By Theorem~\ref{T:n-tilting-t-structure} and Corollary~\ref{C:hom-in-homotopy}, $\Hom_{\D(R)}(A_Z[-1], X)=\Hom_{\K(R)}(A_Z[-1], X)=0$. Hence there is $g\colon A_Z\to Y$ such that $f\circ g=\pi$ showing that $f$ is an epimorphism in $\Modr R$, thus $X$ is isomorphic to an object in $\clH\cap \Modr R=T^\perp.$

In case $Y,Z$ are in $\Add T$ the previous argument shows that $\pi$ splits and so does $f$.

To prove the functoriality note that for every $X\in \clH$, we have a functorial epimorphism $T^{(\Hom_{\clH}(T, X))}\to X\to 0$.

\end{proof}

 Using the theory of derivators as explained in \cite[Section 5]{Sto14} the previous results yields the following:
 \begin{thm}\label{T:derivator-equiv} Let $\clH$ be the heart of the $t$-structure induced by an $n$-tilting $R$-module $T$.
  Then, the inclusion $\clH\subseteq \D(R)$ extends to a triangle equivalence
  \[F\colon \D(\clH)\to \D(R).\]
  \end{thm}
  \begin{proof} It is well known that $T^\perp 
  $ is a functorially coresolving subcategory of $\Modr R$ (see \cite[Ch. 13]{GT12}) and its coresolution dimension is bounded by the projective dimension of $T$.
  
  By Propositions~\ref{P:resol-dimension} and \ref{P:co-resolving}, $T^\perp$ is a functorially resolving subcategory of $\clH$ with resolution dimension bounded by $n$.

By \cite[Proposition 5.14]{Sto14} and  \cite[Remark 5.15]{Sto14} we can argue as in the proof of \cite[Theorem 5.21]{Sto14} to get the conclusion.
\end{proof}

  As noted in Remark~\ref{R:n=1}, in the case of a $1$-tilting module, the objects of the heart $\clH$ can be described in terms of properties of their co-homology modules.
 This is no longer true if $n>1$, but we show a characterization of the complexes in $\clH$ in terms of their cycles and boundaries. The description will be very useful in Section~\ref{S:7}.
 \begin{lem}\label{L:description} Let  $\clH$ be the heart of the  $t$-structure induced by an $n$-tilting module $T$. A complex $X\in \D(R)$ belongs to $
\clH$ if and only if it satisfies the following two conditions:
 \begin{enumerate}
 \item $X$ is quasi isomorphic to a complex \[\dots\to0\to X^{-n}\overset{d^{-n}}\to X^{-n+1}\to \dots \to X^{-1}\overset{d^{-1}}\to X^0\to 0\]  with $X^{-i}\in T^{\perp}$ for all $0\leq i<\leq n$.
 
% For every $-n\leq -i\leq-1$ it holds:
% \begin{itemize}
 \item For every $1\leq i\leq n$, the following hold true:
 \begin{itemize}
 \item[(a)]
 $\Hom_R(T, \Ker d^{-i})= \Hom_R(T, \Img d^{-i-1})$, that is, the trace of $T$ in $\Ker d^{-i}$ coincides with $ \Img d^{-i-1}$.
 \item[(b)] $\Ext^1_R(T, \Ker d^{-i-1})=0$.
 \end{itemize}
  \end{enumerate}
 In particular, if $X\in \clH$ and $H^{-i}(X)=0$ for every $i> j$, then:
 \begin{itemize}
 \item[(i)]
 $X$ is quasi isomorphic to \[\dots\to 0\to X^{-j}\overset{d^{-j}}\to X^{-j+1}\to \dots \to X^{-1}\overset{d^{-1}}\to X^0\to 0\]  with $X^{-i}$ in $T^{\perp}$ for all $0\leq i\leq  j$.
 \item[(ii)] $H^{-j}(X)\in T^{\perp_0}$.
  \end{itemize}

 \end{lem}

\begin{proof} Let $X\in \clH$. By Lemma~\ref{L:U}, we can assume that $X$ is of the form  \[\dots \to X^{-n-1}\to X^{-n}\to X^{-n+1}\to\dots\to X^{-1}\to X^0\to 0\to 0\dots,\] with $X^i$ in $T^\perp $. By Proposition~\ref{P:-n-0-homology} the $-n-1$ truncation of $X$ is an exact complex, so $X$ is isomorphic to 
\[0\to X^{-n}/\Img d^{-n-1}\to X^{-n+1}\to\dots\to X^{-1}\to X^0\to 0\to 0\dots,\] where $X^{-n}/\Img d^{-n-1}$ is in $T^\perp,$ since $\Ker d^{-n-1}\in T^\perp$. This establishes condition (1).

Condition (2) is the translation of the fact that $\Hom_{\K(R)}(T[i], X)=0$, for every $i\neq 0$.

Conversely, it is easy to check that a complex satisfying conditions (1) and (2) belongs to $\clH$.

The last two statements follow easily from (1) and (2). 
%by observing that $tr(\Ker d^n)=0$, by (2)~(a) and $H^{-i}(X)=0$ for every $i> j$ implies that $X$ is quasi isomor
\end{proof}

On the basis of the above description we exhibit some objects of $\clH$ using special $T^\perp$-preenvelopes of $R$-modules.

\begin{nota}\label{N:preenvelopes} If $T$ is an $n$-tilting module and $N\in \Modr R$  is an arbitrary $R$-module, we consider:
\begin{enumerate}
\item (\cite[Ch. 13]{GT12}) A $T^\perp$-coresolution of $N$, that is an exact sequence:
\[0\to N\to B^0\to B^1\to B^2\to \dots \to B^n\to 0,\]
where $B^0$ is a special $T^\perp$-preenvelope of $N$ and for every $i\geq 0$, $B^{i+1}$ is a special $T^\perp$-preenvelope of $A_i=\Coker B^{i-1}\to B^i$ (let $B^{-1}=N$). 
%Thus, $B^i\in \Add T$ for every $i\geq 1$.

\item A short exact sequence $0\to K\to  T^{(\Hom(T, N))}\to\tr_T(N)\to 0$, where $\tr_T(N)$ denotes the trace of $T$ in $N$ and $K\in T^{\perp_1}$.
\end{enumerate}
Moreover, a $T^\perp$-coresolution of $N$ as in (1) can be chosen functorially in $N$, since the cotorsion generated by $T$ is functorially complete
thanks to Quillen' s small object argument and the sequence in (2) is functorial in $N$ by construction.
\end{nota}

\begin{prop}\label{P:N-in-orthogonal}  Let $N$ be an $R$-module. Consider a functorial $T^\perp$-coresolution of $N$ as in Notation~\ref{N:preenvelopes}~(1) and a list of exact sequences as in Notation~\ref{N:preenvelopes}~(2) starting with  $0\to K_{2}\to  T^{(\alpha_{2})}\to tr_T(N)\to 0$ and continuing with $0\to K_{i+1}\to  T^{(\alpha_{i+1} )}\to \tr_T(K_i)\to 0$ for every $i\geq 2$.
Glue them together to construct the complex:
\[X_N=\dots\to T^{(\alpha_{i})}\to\dots\to T^{(\alpha_{2})}\to  B^0\to B^1\to 0,\]
in degrees $\leq 0$ with differentials given by the obvious compositions of the morphisms involved in the short exact sequences. Then $X_N\in \clH$. 

Moreover:
\begin{enumerate}
\item If $N\in T^{\perp_0}$, the complex $X_N=0\to B^0\to B^1\to 0$  (in degrees $-1, 0$) belongs to $\clH$.
\item If $N\in T^{\perp_1}$, the complex \[X_N=\dots\to T^{(\alpha_{n})}\to\dots\to T^{(\alpha_{3})}\to B^0\to B^1\to B^2\to 0,\]
in degrees $\leq 0$ obtained by glueing the short exact sequence $0\to K_3\to  T^{(\alpha_3)}\to \tr_T(N)\to 0$ and the sequences $0\to K_{i+1}\to  T^{(\alpha_{i+1} )}\to \tr_T(K_i)\to 0$ for every $i\geq 3$, belongs to $ \clH$.
\item If $N\in T^{\perp_1}\cap T^{\perp_2}\cap\dots \cap T^{\perp_{n-1}}$, the complex
\[X_N=\dots\to T^{(\alpha_{i})}\to\dots\to T^{(\alpha_{n+1})}\to  B^0\to B^1\to B^2\to \dots \to B^n\to 0,\]
in degrees $\leq 0$ obtained by glueing the short exact sequences $0\to K_{n+1}\to  T^{(\alpha_{n+1})}\to \tr_T(N)\to 0$ and $0\to K_{i+1}\to  T^{(\alpha_{i+1} )}\to \tr_T(K_i)\to 0$ for every $i\geq n+1$, belongs to $ \clH$.
\end{enumerate}
\end{prop}

\begin{proof} Follows easily by the characterization of the complexes in $\clH$ stated in Lemma~\ref{L:description}.
\end{proof}
We can apply the previous proposition to obtain information about the torsion radical of the torsion pair induced by an $n$-tilting  module $T$.
\begin{cor}\label{C:torsion} Let $T$ be an $n$-tilting  module $T$ and consider the torsion pair $(^{\perp_0}(T^{\perp_0}), T^{\perp_0})$ associated to $T$.

 If $N\in T^{\perp_1}\cap T^{\perp_2}\cap\dots \cap T^{\perp_{n-1}}$, then the torsion submodule of $N$ in the torsion pair  is given by $\tr_T(N)$. Moreover, $\tr_T(N)\in T^{\perp_{n}}\cap T^{\perp_{n-1}}$ and
   $N/\tr_T(N)\in T^{\perp_0}\cap T^{\perp_{n-1}}$. 
   
   In particular, if $n=2$ and $N\in T^{\perp_{1}}$, $\tr_T(N)\in T^\perp$ and $N/\tr_T(N)\in  T^{\perp_0}\cap T^{\perp_{1}}$.
   \end{cor}
   \begin{proof} Given $N$, consider the complex $X_N$ constructed in Proposition~\ref{P:N-in-orthogonal}~(3). Since $X_N\in \clH$, $H^{n}(X)\in T^{\perp_0}$, by Lemma~\ref{L:description}~(ii) and by construction $H^{-n}(X)\cong N/\tr_T(N)$.
   \end{proof}   
   
Another application of Proposition~\ref{P:N-in-orthogonal} is given by the following:   
   \begin{prop}\label{P:coproducts} The heart $\clH$  is closed under coproducts in $\D(R)$ if and only if  $T^{\perp_1}$ is closed under direct sums in $\Modr R$.
  
 \end{prop} 
 \begin{proof} Let $(X_{\alpha}: \alpha\in \Lambda)$ be a family of objects of $\clH$. Up to isomorphisms the $X_{\alpha}$' are represented by complexes in $\Ch (R)$ as described in Lemma~\ref{L:description}. Let $X$ be the coproduct of the $X_{\alpha}$' in $\Ch (R)$. The cycles and boundaries of the complex $X$ are the coproducts of the cycles and the boundaries of the complexes $X_{\alpha}$. Thus if $T^{\perp_1}$ is closed under direct sums, $X$ satisfies conditions (1) and ~(2) of Lemma~\ref{L:description}, hence $X\in \clH$.
 
 Conversely, assume that $\clH$ is closed under coproducts in $\D(R)$ and let $(N_{\alpha}: \alpha\in \Lambda)$ be  a family of modules in $T^{\perp_1}$.
  For each $\alpha$ consider the complex $X_{N_{\alpha}}\in \clH$ constructed in Proposition~\ref{P:N-in-orthogonal}~(2). By assumption the coproduct in $\D(R)$ of the $X_{N_{\alpha}}$'s belongs to $\clH$ and, again by Lemma~\ref{L:description} conditions (1) and (2), we conclude that $\Ext^1_R(T, \oplus N_{\alpha})=0$.
  \end{proof}
\begin{rem}\label{R:coproducts} The condition that $T^{\perp_1}$ be closed under direct sums is automatically true for a $1$-tilting module $T$, but in general it is not true for an $n$-tilting module with $n>1$.
\end{rem}
%
%\begin{prop}\label{P:heart} Let $T$ be a tilting module of projective dimension $n$. Let $X$ be a complex in the heart $\clH$ of the  $t$-structure induced by $T$. Then
%$\H^i(X)= 0$ for every $i<-n$ and $i>0$.
%\end{prop}
%%
%\begin{proof} By \cite[Proposition 3.5]{B04} we can chose a sequence as in (T3) of Definition~\ref{D:n-tilting} with $r=n$. Apply the functor$ \Hom_{\D(R)}(T[i], -)$ to the triangles in $\D(R)$ corresponding to the short exact sequences in which the exact sequence (T3) splits, we conclude by the characterization of the objects in $\clH$ as in Theorem~\ref{T:n-tilting-t-structure}.\end{proof}

%From now on $\clH$ will denote the heart of the $t$-structure induced by an $n$-tilting module $T$. 
\section{The heart $\clH$ and the module category over $\End(T)$}~\label{s:End(T)}

We will make use of the results about derived equivalence induced by {\sl good} $n$-tilting modules proved in \cite{BMT} and \cite{BP}.  

\begin{defn}\label{D:good} An $n$-tilting module $T_R$ is \emph{good} if the terms in the exact sequence (T3) in Definition~\ref{D:n-tilting}can be chosen to be direct summands of finite direct sums of copies of $T$. 
By \cite[Proposition 1.3]{BMT} every $n$-tilting module $T_R$ is equivalent to a good $n$-tilting module. \end{defn}

We recall the following facts about good $n$-tilting modules.
\begin{fact}\label{F:good-tilting}\emph{Let $T$ is  a good $n$-tilting module with $S=\End(T_R)$. The following hold:
\begin{enumerate}
\item
 (\cite{BMT} and \cite{Miy}) 
 \begin{enumerate}
 \item[(a)] p.d$_ST\leq n$ and $_ST$ has a finite projective resolution consisting of finitely generated projective left $S$-modules.
\item[(b)] $\Ext^i_S(T, T^{(\alpha)})=0$ for every $i\geq 1$.
\end{enumerate}
\item (\cite[Theorem 2.2]{BMT}, \cite[Proposition 5.2]{BP})
 \begin{enumerate}
\item[(i)] The pair $(\mathbb L G, \mathbb R H)$:
\[\xymatrix{\D(R) \ar@/_1pc/
[rrr]_{\mathbb R H=\RHom_R(T, -)} &&&\D(S)  \ar@/_1pc/[lll]_{\mathbb L G= -\Lotimes_S T}} 
\]
%\[\xymatrix{\D(R) \ar@/^2pc/
%[rrr]_{\mathbb R H=\RHom_R(T, -)}\ar@/_2pc/ [lll]^{\mathbb L G= -\Lotimes_S T} &&& \D(R) \ar@/^2pc/
%[lll]^{\RHom_R(T, -)} \ar@/_2pc/ [lll]}
%\]
% \[\mathbb R H=\RHom_R(T, -)\colon \D(R)\to \D(S),\\ \qquad
%\mathbb L G= -\Lotimes_S T\colon \D(S)\to \D(R)\] 
 is  an adjoint pair;
\item[(ii)]  the functor $\mathbb R H:\D(R)\to\D(S)$ is fully faithful;
\item[(iii)] the essential image of $\RHom _R(T, -)$ is $\Ker  (\mathbb L G)^\perp$ where \[\Ker  (\mathbb L G)^\perp= \{Z\in \D(S)\mid \Hom_{\D(S)}(Y, Z)=0, \textrm{\ for\ all\ } Y\in \Ker  ( \mathbb L G)\},\] 
\end{enumerate}
\end{enumerate}}

\end{fact}

We illustrate a property of the functor $ \mathbb L G$ which will be useful in Section~\ref{s:dir-lim-heart}.
 \begin{lem}\label{L:counit} Let
Let  $(M_{\alpha}; g_{\beta\alpha})_{\alpha\in \Lambda}$ be a  direct system of right $S$-modules. There are   projective resolutions $P_{\alpha}$  of $M_{\alpha}$ such that the direct system $(M_{\alpha}; g_{\beta\alpha})$ can be lifted to a direct system $(P_{\alpha};\tilde{g}_{\beta\alpha})_{\alpha\in \Lambda}$ in $\Ch(S)$ giving rise to the following isomorphisms in $
\D(R)$:
\[ M_{\alpha}\Lotimes_ST\cong P_{\alpha}\otimes_ST;\, \varinjlim\limits_{\Ch (R)}(P_{\alpha}\otimes_ST)\cong (\varinjlim\limits_{\Ch (S)}P_{\alpha})\otimes_ST)\cong (\varinjlim\limits_{\Modr S}M_{\alpha})\Lotimes_ST\].
\end{lem}
\begin{proof} 
The observation that for every module $A\in \Modr S$ the canonical epimorphism $S^{(\Hom_S(S, A))}\to A$ is functorial in $A$, implies that for each $M_{\alpha}\in \Modr S$ we can choose functorially a projective resolution $P_{\alpha}$ so that the direct system $(M_{\alpha}; g_{\beta\alpha})_{\alpha\in \Lambda}$ in $\Modr S$ can be lifted to a direct system $(P_{\alpha};\tilde{g}_{\beta\alpha})_{\alpha\in \Lambda}$ in $\Ch(S)$. We have $M_{\alpha}\Lotimes_ST\cong P_{\alpha}\otimes_ST$ and, since $\varinjlim\limits_{\Ch (S)}P_{\alpha}$ is a flat resolution of $M{}={}\varinjlim\limits_{\Modr S} M_{\alpha}$ we also get $ M\Lotimes_ST\cong (\varinjlim\limits_{\Ch (S)}P_{\alpha})\otimes_ST$. Thus the following isomorphisms hold in $\D(R)$:
\[\varinjlim\limits_{\Ch (R)}(P_{\alpha}\otimes_ST)\cong (\varinjlim\limits_{\Ch (S)}P_{\alpha})\otimes_ST\cong (\varinjlim\limits_{\Modr S}M_{\alpha})\Lotimes_ST.\]
\end{proof}

{\bf From now on in this section, $\clH$ will always denote the $t$-structure induced by a good $n$-tilting module $T_R$ with endomorphism ring $S$.}

%
%\begin{thm}\label{T:BP}(\cite[Proposition 5.2]{BP}) In the notation of Theorem~\ref{T:BMT} the functors $\RHom _R(T, -)$ and $ -\Lotimes_S T$ induce a recollement of derived categories:
%
%\[\xymatrix{\D(E) \ar[rrr] &&&\D(S) \ar@/^2pc/
%[lll]\ar@/_2pc/ [lll] \ar[rrr]^{-\Lotimes_S T} &&& \D(R) \ar@/^2pc/
%[lll]^{\RHom_R(T, -)} \ar@/_2pc/ [lll]}
%\]
%
%where $E$ is a suitable dg-algebra. 
%
%
%Thus, if \[\Ker  (\mathbb L G)^\perp= \{Z\in \D(S)\mid \Hom_{\D(S)}(Y, Z)=0, \textrm{\ for\ all\ } Y\in \Ker  ( \mathbb L G)\},\] then $\Ker  (\mathbb L G)^\perp$ is the essential image of $\RHom _R(T, -)$.
%\end{thm}
%%
A characterization of the objects in $\clH$ is given by the following:
\begin{lem}\label{L:hom-in-degree-0}  A complex in $\D(R)$ belongs to $\clH$ if and only if it is isomorphic to a complex $X$ with terms in the tilting class $T_R^\perp$ such that $\RHom_R(T, X)$ has cohomology concentrated in degree zero and isomorphic to $\Hom_{\D(R)}(T, X)$ (and thus also to $\Hom_{\K(R)}(T, X)$ by Corollary~\ref{C:hom-in-homotopy}).
\end{lem}
\begin{proof} By Lemma~\ref{L:fibrant} every complex in $\D(R)$ is isomorphic to a complex $X$ with terms in $T^\perp$. Thus, $X$ is a $\Hom_R(T, -)$-acyclic object and since p.d.$T\leq n$, $\RHom_R(T, X)\cong\Hom_R(T, X)$ (see e.g. \cite[Theorem I.5.1]{Hart}). Moreover, $H^{-i}(\Hom_R(T, X)\cong \Hom_{\K(R)}(T[i], X)$ and by Corollary~\ref{C:hom-in-homotopy},  $\Hom_{\K(R)}(T[i], X)\cong \Hom_{\D(R)}(T[i], X)$. 

Hence if $X\in \clH$, then $\RHom_R(T, X)$ has cohomology concentrated in degree zero and it is isomorphic to $\Hom_{\K(R)}(T, X)\cong \Hom_{\D(R)}(T, X)$. 

Conversely, if $\RHom_R(T, X)\cong \Hom_{\D(R)}(T, X)$, then $\Hom_{\K(R)}(T[i], X)=0$ for every $i\neq 0$, hence $X\in \clH$.
\end{proof}
 \begin{prop}\label{P:left-adjoint} %Let $\clH$ the heart of the $t$-structure induced by a good $n$-tilting module $T_R$ with endomorphism ring $S$. 
The following condition hold true:
 \begin{enumerate}
 \item The functor
$H_T= \Hom_{\clH}(T, -)\colon \clH\to \Modr S$ is exact and fully faithful.
\item The essential image of $ \Hom_{\clH}(T, -)$ is given by $\Ker  (\mathbb L G)^\perp\cap \Modr S.$
\item $H_T$ has a left adjoint $F$ given by $b\circ G$ where $G$ is the restriction to $\Modr S$ of the functor $\mathbb L G$ and $b$ is left adjoint of the inclusion functor $\iota\colon\clH\to \U$.
\item There is an equivalence $\clH\cong \Modr S[\Sigma^{-1}]$\\
 where $\Sigma= \{ g\in \Modr S\mid F(g) \textrm {\ is\ an\ isomorphism}\}.$\end{enumerate}
\end{prop}
\begin{proof} (1) The functor $\Hom_{\clH}(T, -)$ has image in $\Modr S$, by Lemma~\ref{L:hom-in-degree-0} and it is exact, since $T$ is a projective object of $\clH$, by Proposition~\ref{P:-n-0-homology}.
Let $X\in \clH$; we can assume that $X$ has terms in $T^\perp$. By Lemma~\ref{L:hom-in-degree-0} $\RHom _R(T, X)$ is  isomorphic to $\Hom_{\D(R)}(T, X)$ and, $\Hom_{\D(R)}(T, X)\cong\Hom_{\clH}(T, X)$, since $\clH$ is a full subcategory of $\D(R)$.
Thus the functor $H_T=\Hom_{\clH}(T, -)$ is isomorphic to the restriction at $\clH$ of the functor $\RHom_R(T, -)$, hence $H_T$ is fully faithful, by Fact~\ref{F:good-tilting}~(ii).

(2) Easily follows by Fact~\ref{F:good-tilting}~(iii) and Lemma~\ref{L:hom-in-degree-0}.
 
(3) The functor $\mathbb  L G$ is left adjoint to $\mathbb R H$ and for every right $S$-module $M_S$,  $\mathbb  L G(M)\cong P_M\otimes _ST$ where $P_M$ is a projective resolution of $M_S$. Thus, $\mathbb  L G(M)$ is isomorphic to a complex of $R$-modules with terms in $\Add T$ and in degrees $i\leq 0$. By Lemma~\ref{L:U} ,  $\mathbb  L G(M)\in \U$. By Remark~\ref{R:Saorin} the inclusion $\iota\colon \clH\to \U$ admits a left adjoint $b$. 
We show now that the functor $F=b\circ G$, where $G$ is the restriction of $\mathbb  L G$ to $\Modr S$, is left adjoint to $H_T$.
Let $X\in \clH$ and $M\in \Modr S$, then
\begin{align*}&\Hom_S(M, H_T(X))\cong \Hom_{\D(S)}(M, \RHom_R(T, X))\cong \Hom_{\D(R)}(G(M), X)\cong\\
&\cong \Hom_{\U} ((G(M), X)\cong \Hom_{\clH}(b(G(M), X).\end{align*}

(4) Follows by \cite[Proposition 1.3]{GZ67}. \end{proof}

The situation described by Proposition~\ref{P:left-adjoint} can be depicted by the following diagram:
 \[\xymatrix{\D(R)\ar@/_/[rr]_{\RHom _R(T, -)}&&\D(S)\ar@/_/[ll]_{ -\Lotimes_S T}\\
 \U\ar@{^(->}[u]\ar@/_1pc/[d]^b\\
 \clH\ar@{^(->}[u]_{\iota}\ar[rr]_{\Hom_{\clH}(T, -)}&&\Modr S\ar@{^(->}[uu]_{can}\ar@/_/[llu]_G
 }\]

By Proposition~\ref{P:left-adjoint}~(1) and (2), the functor $\Hom_{\clH}(T, -)$ induces an equivalence between $\clH$ and \[\Ker  (\mathbb L G)^\perp\cap \Modr S.\]

Fort the rest of this section we deal with our main concern which is to characterize the case in which the heart $\clH$ of the $t$-structure induced by a good $n$-tilting module is a Grothendieck category.
In Theorem~\ref{T:E-hereditary} and Theorem~\ref{T:S/A} we will give characterizations in terms of properties of subcategories of $\Modr S$.

 A first observation is obtained by an application of the Gabriel-Popesco's Theorem (\cite{GP}).
 
 \begin{prop}\label{P:GP} Let $F\colon \Modr S\to \clH$ be the left adjoint of the functor $\Hom_{\clH}(T, -)$ given by Proposition~\ref{P:left-adjoint}~(3).
  The following are equivalent:
  \begin{enumerate}
  \item $\clH$ is a Grothendieck category;
  \item $F$ is an exact functor;
  \item $\Ker F$ is a hereditary torsion class in $\Modr S$.
  \end{enumerate}
  \end{prop}
  \begin{proof} (1) $\Rightarrow $ (2) follows from Gabriel-Popesco's Theorem~\cite{GP}, since $T$ is a generator of $\clH$ by Proposition~\ref{P:resol-dimension}.
    
  (2) $\Rightarrow $ (3) From (2) it follows that $\Ker F$ is a Serre subcategory of $\Modr S$, that is for every short exact sequence 
  $0\to N\to M \to L\to 0$ in $\Modr S$, $M\in \Ker F$ if and only if $N$ and $L$ are in $\Ker F$. Since $F$ is a left adjoint, it sends coproduces in $\Modr S$ to coproducts in $\clH$. So $\Ker F$ is a hereditary torsion class.
  
  (3) $\Rightarrow $ (1) As in Proposition~\ref{P:left-adjoint}, let \[\Sigma= \{ g\in \Modr S\mid F(g) \textrm {\ is\ an\ isomorphism}\}. \] When $\Ker F$ is a hereditary torsion class, then by \cite[Chap. III]{G62}, $g\in \Sigma$ if and only if $\Ker g$ and  $\Coker g$ belong to $\Ker F$. Thus, \[\Modr S[\Sigma^{-1}]\cong\Modr S/\Ker  F\] and the latter category is well known to be a Grothendieck category.
   The conclusion follows by Proposition~\ref{P:left-adjoint}~(4).
  \end{proof}

If $\C$ is a subcategory of an abelian category $\A$, its perpendicular category, denoted by $\C_{\perp}$, is defined by:
\[\C_{\perp}=\{X\in \A\mid \Hom_{\A}(C, X)=\Ext^1_{\A}(C, X)=0, \textrm{ for\ all\ } C\in \C\}.\]

We define also 
\[\C_{\perp_\infty}=\{X\in \A\mid \Ext^i_{\A}(C, X)=0, \textrm{ for\ all\ } C\in \C\ \textrm{ for\ all\ }i\geq 0\}.\]
\begin{defn}\label{D:E} If $T_R$ is a good $n$-tilting module with endomorphism ring $S$ we let \[\E=\{M_S\in \Modr S\mid \Tor^S_i(M, T)=0, \textrm{\ for\ all\ } i\geq 0\}.\]
 \end{defn}

\begin{lem}\label{L:perp-infinity} %Let $\clH$ be the heart of the $t$-structure induced by a good $n$-tilting module $T_R$ with endomorphism ring $S$. 
The essential image of the functor $\Hom_{\clH}(T, -)$ is contained in $\E_{\perp_\infty}$ and they coincide in case $\E$
%($\E=\{M_S\in \Modr S\mid Tor^S_i(M, T)=0,\forall i\geq 0\}.$ )
 is a hereditary torsion class in $\Modr S$.
\end{lem}
\begin{proof} Let $M\cong \Hom_{\clH}(T, X)$ for some $X\in\clH$. For every $E\in \E$ and every $j\in \mathbb Z$, $E[j]$ belongs to $\Ker (\mathbb L G)$,  hence by Proposition~\ref{P:left-adjoint}~(2), $\Hom_{\D(S)}(E[j], M)=0$. But $\Hom_{\D(S)}(E[j], M)\cong\Ext^{-j}_S(E, M)$, for every $j\in \mathbb Z$, hence $M\in\E_{\perp_{\infty}}$.

To prove the other statement we have to show that, if $\E$ is hereditary and $M\in\E_{\perp_{\infty}}$,  then $\Hom_{\D(S)}(Y, M)=0$ for every $Y\in \Ker (\mathbb L G)$.
Note that a complex belongs to $\Ker  (\mathbb L G)$ if and only if it is quasi isomorphic to a complex with terms in $\E$: This has been proved in \cite[Proposition 4.6]{CX} for the case of a good $1$-tilting module, but using \cite[Lemma 1.5]{BMT} everything goes through for the $n>1$ case. The cycles and the boundaries of every complex with terms in $\E$ are again in $\E$, since we are assuming that $\E$ is a hereditary torsion class.

(a)  We first prove that if $Z$ is a bounded complex with terms in $\E$ and $M\in \E_{\perp_\infty}$, then $\Hom_{\D(S)}(Z, M)=0$.
We make induction on the number $k$ of non-zero terms in $Z$. If $k=1$, $Z$ is of the form $E[j]$ for some $j\in \mathbb Z$ and some $E\in \E$. Thus, $\Hom_{\D(S)}(E[j], M)\cong \Ext^{-j}_S(E, M)=0$. Let now $Z=0\to E^i\to E^{i+1}\to\dots\to E^{i+k}\to 0 $ and let $K$ be the kernel of the $i^{\textrm{th}}$-differential of $Z$. Since $\E$ is hereditary,  we have a triangle  $K\to Z\to Z'\to K[1]$ where $K\in \E$ and $Z'$ is a bounded complex with at most $k-1$-terms in $\E$. Thus, by induction $\Hom_{\D(S)}(Z', M)=0$, hence also $\Hom_{\D(S)}(Z, M)=0$.

(b) Let now $Y\in \Ker  (\mathbb L G)$ be a bounded below complex. Then $Y$ is a homotopy colimit of its trunctation subcomplexes $Z_n$ which are bounded and with terms in $\E$, again by the hereditary condition on $\E$. Hence, from the triangle $\coprod_iZ_i\to\coprod_iZ_i\to Y\to \coprod_iZ_i[1]$ and by (a) we conclude that $\Hom_{\D(S)}(Y, M)=0$.
It remains to consider the case of a bounded above complex $Y\in \Ker  (\mathbb L G)$. (Note that $Y$ is a homotopy limit of its quotient complexes obtained from truncations, which are bounded and with terms in $\E$, but the triangle of the homotopy limit doesn't help to conclude).

(c) To prove that  $\Hom_{\D(S)}(Y, M)=0$ for a bounded above complex, we consider a suitable model structure on $\D(S)$ described as follows. Let $W$ be an injective cogenerator of $\Modr R$ and let $(-)^d=\Hom_R(-, W)$ denote the dual on any right $R$-module. Then $T^d=C$ is a pure injective right $S$-module and by well known homological formulas we have $\Ext^i_S(N, C)\cong [\Tor^S_i(N, T)]^d$, for every right $S$-module $N$ and every $i\geq 0$. Hence $\E\subseteq {}^\perp C$. We consider the model structure on $\D(S)$ induced by the complete cotorsion pair $(^\perp C, (^\perp C)^\perp)$  as described in Theorem~\ref{T:C-model-structure}.
Let $Y$ be a bounded above complex with terms in $\E$ (hence $Y \in \Ker  (\mathbb L G)$). By Theorem~\ref{T:C-model-structure}, $Y$ is a cofibrant object in the model structure induced by $C$. For every $N\in \Modr S$ a fibrant replacement of $N$ is a $(^\perp C)^\perp$-coresolution of $N$ constructed by taking special $(^\perp C)^\perp$-preenvelopes. Now, let $M\in\E_{\perp_{\infty}}$ and let $I$ be its fibrant replacement. By Corollary~\ref{C:hom-in-homotopy-dual}, $\Hom_{\D(S)}(Y, M)\cong \Hom_{\K(S)}(Y, I)$ and by (a) and the hereditary condition on $\E$ we can assume that $Y$ has non zero terms only in degrees $i\leq 0$. If $f\colon Y\to I$ is a cochain map we have
 \[\xymatrix{0\ar[r]&\dots\ar[r]&Y^{-2}\ar[r]&Y^{-1}\ar[r]\ar[d]&Y^0\ar[r]\ar[d]_{f^0}&0\ar[d]\ar[r]&\dots\\
 &&\dots\ar[r]&0\ar[r]&I^0\ar[r]^{d^0}& I^1\ar[r]&I^2\ar[r]&\dots },\]
where $\Img f^0\subseteq \Ker d^0=M$. Since $\Hom_S(E, M)=0$ for every $E\in \E$, we conclude that  $f^0=0$ and thus also $\Hom_{\D(S)}(Y, M)=0$.
\end{proof}

\begin{prop}\label{P:kerF=E} %Let $\clH$ be the heart of the $t$-structure induced by a good $n$-tilting module $T_R$ with endomorphism ring $S$. 
Assume that $\clH$ is a Grothendieck category .Let $\E$ be  as in Definition~\ref{D:E} and let $F$ be the left adjoint of the functor $\Hom_{\clH}(T, -)$ given by Proposition~\ref{P:left-adjoint}~(3). Then, the following hold true:
\begin{enumerate}
\item $\Ker F$ is a hereditary torsion class and $\Ker F=\E$.
\item $(\Ker F)_{\perp}$ is the essential image of the functor $\Hom_{\clH}(T, )\colon \clH \to \Modr S.$
\item $\E_{\perp_\infty}=\E_{\perp}=(\Ker F)_{\perp}$.
\end{enumerate}
 \end{prop}
 \begin{proof}
 (1) $\Ker F$ is a hereditary torsion class by Proposition~\ref{P:GP}. The inclusion $\E\subseteq \Ker F$ is immediate, since if $M\in \E$, then $\mathbb L G(M)=0$, hence $F(M)=0$.
 
 For the converse, let $M\in \Ker F$ and let \[P_M= \dots P^{-i}\overset{d^{-i}}\to P^{-i+1}\to \dots \to P^{-1}\overset{d^{-1}}\to P^0\to M\to 0\] be a projective resolution of $M$. By Proposition~\ref{P:GP},  $F$ is an exact functor, hence the sequence 
 \[\dots F(P^{-i})\overset{F(d^{-i})}\longrightarrow F(P^{-i+1})\to \dots \to F(P^{-1})\overset{F(d^{-1})}\longrightarrow F(P^0)\to F(M)=0\to 0\] is exact in $\clH$ and $F(P^{-i})\cong P^{-i}\otimes_ST$ belongs to $\Add T$, so that $F(d^{-i})$ is naturally isomorphic to $d^{-i}\otimes_S1_T$. We infer that $P_M\otimes_ST=0$, hence $M\in \E$.
 
 (2) By (1) and by \cite[Chap III]{G62}, the canonical quotient functor $q\colon \Modr S\to \Modr S/\Ker F$ is exact and it admits a fully faithful right adjoint $a$ whose essential image is the perpendicular category $(\Ker F)_{\perp}$ consisting of the closed objects. By Proposition~\ref{P:left-adjoint} we have the following diagram
  \[\xymatrix{\clH\ar@/_/[rr]_{\Hom _{\clH}(T, -)}&&\Modr S\ar@/_/[ll]_{ F}\ar@/_/[d]^q\\
&&\Modr S/\Ker F\ar@/^1pc/[llu]^{F'}\ar@/_/[u]_{a} }\]
 where $F'$ is the unique functor such that $F'\circ q= F$ and $F'$ is an equivalence of categories. Let $D$ be an inverse of $F'$; then $(F', D)$ is an adjoint pair, so that $(F, aD)$ is an adjoint pair. Hence, the functor $H_T=\Hom_{
\clH}(T, -)$ is naturally isomorphic to the functor $aD$ and thus also $a\cong H_T\circ F'$. We conclude that the essential images of $a$ and $H_T$ coincides and so they coincide with $(\Ker F)_{\perp}$.

(3) Clearly $\E_{\perp_\infty}\subseteq \E_{\perp}$. In view of conditions (1) and (2), to show the reverse inclusion it is enough to prove that the essential image of the functor $\Hom_{\clH}(T, )$ is contained in $\E_{\perp_\infty}$. Let $M_S$ be of the form $\Hom_{\clH}(T, X)$ for some $X\in \clH$. By Proposition~\ref{P:left-adjoint}~(2) we have that $\Hom_{\D(S)}(Y, M)=0$ for every $Y\in \Ker (\mathbb L G)$. In particular, for every $E\in \E$, $\Hom_{\D(S)}(E[j], M)=0$, for every $j\in \mathbb Z$; hence $\Ext^j_S(E, M)=0$ for every $j\geq 0$, that is $M\in \E_{\perp_\infty}$.
\end{proof}

\begin{thm}\label{T:E-hereditary} %Let $\clH$ be the heart of the $t$-structure induced by a good $n$-tilting module $T_R$ with endomorphism ring $S$ and 
Let $\E$ be  as in Definition~\ref{D:E}. The following are equivalent:
\begin{enumerate}
\item $\clH$ is a Grothendieck category.
\item $\E$ is a hereditary torsion class in $\Modr S$ and $\E_{\perp}=\E_{\perp_\infty}$.
\item For every $M\in \Modr S$, $\mathbb L G(M)\in \clH$.
\item If $G$ is the restriction of  $\mathbb L G $ to $\Modr S$, then $G$ is naturally isomorphic to the left adjoint $F$ of the fully faithful functor $\Hom_{\clH}(T, -)\colon \clH\to \Modr S$ constructed in Proposition~\ref{P:GP} and $G$ is an exact functor.
\end{enumerate}
\end{thm}
\begin{proof} (1) $\Rightarrow $ (2) Follows by Proposition~\ref{P:kerF=E}.

(2) $\Rightarrow $ (3) Let $M\in \Modr S$ and let $M_{\E}$ be the torsion submodule of $M$ with respect to the torsion pair $(\E, \E^{\perp_0})$. Then, clearly $\mathbb L G(M/M_{\E})\cong \mathbb L G(M)$ and so, w.l.o.g. we may assume that $M$ is $\E$-torsion free.

By \cite[Proposition 2.2]{GL91} $M$ has a $\E_{\perp}$-reflection. That is there is a short exact sequence
$0\to M\to Y\to E\to 0$ with $Y\in \E_{\perp}$ and $E\in \E$. Thus $\mathbb L G(Y)\cong 
\mathbb L G(M)$. Now, by assumption and Lemma~\ref{L:perp-infinity}, there is $X\in \clH$ such that $\Hom_{\clH}(T, X)\cong Y$. By  Lemma~\ref{L:hom-in-degree-0} $\Hom_{\clH}(T, X)\cong \Hom_{\D(R)}(T,X)\cong \RHom _R(T, X)$, hence $\mathbb L G(Y)\cong \mathbb L G( \RHom _R(T, X))\cong X$, since $\Hom_{\clH}(T, -)$ is fully faithful. We conclude that $\mathbb L G(M)\in \clH$.

(3) $\Rightarrow $ (4) Let $F$ be the left adjoint of the functor $\Hom_{\clH}(T, -)\colon \clH\to \Modr S$ given by Proposition~\ref{P:left-adjoint}~(3). Condition (3) implies that for every $M\in \Modr S$, $F(M)\cong G(M).$  To show that $G$ is an exact functor we prove that $\Ker G=\Ker F$ is a hereditary torsion class in $\Modr S$.
 Since $G$ is right exact it is enough to show that $\Ker G$ is closed under submodules.
Let $0\to M\to L\to N\to 0$ be an exact sequence in $\Modr S$ with $L\in \Ker G$. From the triangle $N[-1]\to M\to L\to N$ we have the triangle \[N\Lotimes_ST[-1]\to M\Lotimes_ST\to L\Lotimes_ST\to N\Lotimes_ST\] in $\D(R)$. By assumption $L\Lotimes_ST=G(L)=0$, and $N\Lotimes_ST=G(N)=0$. Then, also $N\Lotimes_ST[-1]$ is zero, and by the above triangle we conclude that $M\Lotimes_ST$ is zero. Hence $G(M)=0$.

(4) $\Rightarrow $ (1) Follows by  Proposition~\ref{P:GP}.
\end{proof}
 We can now interpret the previous characterization in terms of homological epimorphisms and ``generalized universal localization" 
 which is a generalization of the well know concept of universal localization in Schofield's sense (\cite{Scho}).
\begin{defn}
\begin{enumerate}
\item[(1)] (\cite {GL91} A ring homomorphism $f\colon S\to U$ is a homological ring epimorphism if the associated restriction functor $f_{\ast}\colon \D(U)\to \D(S)$ is fully faithful. Equivalently, $f\colon S\to U$ is a ring epimorphism and $\Tor^S_i(U,U)=0$ for every $i\geq 1$.
\item[(2)] (\cite[Section 15]{Kr05}) Let $S$ be a ring and $\Sigma$ a set of
perfect complexes $P\in \K(S^{\mathrm{op}})$. 
A ring $U$ is a \emph{generalized universal localization} of $S$ at the set $%
\Sigma$ if there is a ring homomorphism $\lambda\colon S\rightarrow U$ such that $%
U\underset{S}\otimes P$ is acyclic and $\lambda$ satisfies the universal property with respect to this property. That is, 
 for every ring homomorphism $\mu\colon S\rightarrow R$ such that $%
R\underset{S}\otimes P$ is acyclic, there exists a unique ring homomorphism $%
\nu\colon U\rightarrow R$ such that $\nu\circ \lambda=\mu$.
\end{enumerate}
\end{defn}

\begin{thm}\label{T:S/A} Let $\clH$ be the heart of the $t$-structure induced by a good $n$-tilting module $T_R$ with endomorphism ring $S$ and let $\E$ be  as in Definition~\ref{D:E}. The following are equivalent:
\begin{enumerate}
\item $\clH$ is a Grothendieck category.
\item There is an idempotent two sided-ideal $A$ of $S$ projective as a right $S$-module such that $\E=\Modr S/A$.
\end{enumerate}
If the above conditions are satisfied, then the canonical morphism $S\to S/A$ is a homological epimorphism and $S/A$ is a generalized universal localization at  a projective resolution of $_ST$.

In particular, there is recollement:

\begin{equation*}
\xymatrix{\D(S/A) \ar[rrr]^{\lambda_{\ast}} &&&\D(S) \ar@/^2pc/
[lll]^{i^{!}=\RHom_S(S/A, -)} \ar@/_2pc/ [lll]_{i^{\ast}=
-\overset{\mathbb L} {\underset{S}{\otimes}}(S/A)} \ar[rrr]^{j^{\ast}=
-\Lotimes_ST} &&& \D(R) \ar@/^2pc/ [lll]^{j_{\ast}=\RHom_R(T, -)} \ar@/_2pc/ [lll]_{j_{!}} }
\end{equation*}

%\end{enumerate}
\end{thm}
\begin{proof}
(1) $\Rightarrow$ (2) By Theorem~\ref{T:E-hereditary}  $\E$ is a hereditary torsion classe. Since $_ST$ has a projective resolution consisting of finitely generated projective $S$-modules, $\E$ is closed under direct products. Thus $\E$ is a torsion-torsion free class. By \cite[Proposition 6.11]{Ste75} there is an idempotent two sided ideal $A$ of $S$ such that $\E=\Modr S/A$. By \cite[Theorem 6.1]{BP} the canonical morphism $S\to S/A$ is a homological epimorphism, hence by \cite[Theorem 4.4]{GL91}, $\Ext^i_S(S/A, E)=0$ for every $E\in \E$ and every $i\geq 1$. We show now that $A_S$ is moreover a projective module.

By Theorem~\ref{T:E-hereditary}, we know that  $\E_{\perp}=\E_{\perp_\infty}$. Hence for every module $Y\in \E_{\perp}$, $\Ext^i_S(S/A, Y)=0$ for every $i\geq 1$. Let $M\in \Modr S$ and let $M_{\E}$ be the $\E$-torsion submodule of $M$. As in the proof of Theorem~\ref{T:E-hereditary} (3) $\Rightarrow $ (4), consider an $\E$-reflection of $M/M_{\E}$ that is a module $Y\in  \E_{\perp}$ such that there is a short exact sequence
$0\to M/M_{\E}\to Y\to E\to 0$ with $E\in \E$. By the above remarks we have that $\Ext^i_S(S/A, M/M_{\E})=0$ for every $i\geq 2$ and from the exact sequence $0\to M_{\E}\to M\to M/M_{\E}\to 0$ we conclude that $\Ext^i_S(S/A, M)=0$, for every $i\geq 2$, hence p.dim $S/A\leq 1$ and $A$ is projective as a right $S$-module.

(2) $\Rightarrow$ (1) $\E=\Modr S/A$ implies that $\E$ is a hereditary torsion class and the projectivity of $A_S$ implies tat $\E_{\perp}=\E_{\perp_\infty}$. Hence the conclusion follows by Theorem~\ref{T:E-hereditary}.

The last statement follows by \cite[Theorem 6.1, Proposition 7.3 ]{BP}.
\end{proof}

For the case of a good $1$-tilting module, condition (2) in Theorem~\ref{T:E-hereditary} can be weakened, since the assumption on $\E$ to be hereditary is enough. 
To see this we use the following result  \cite[Section 4]{B10}.

\begin{rem}\label{R:C-torsion} (\cite[Section 4]{B10}.) If  $_ST_R$ is a good $1$-tilting module and $W$ is an injective cogenerator of $\Modr R$, we let  $T^d=\Hom_R(T, W)$ be the dual of $T$. Then 
%$T^d=C$ is a pure injective right $S$-module and by well known homological formulas we have $\Ext^i_S(N, C)\cong [\Tor^S_i(N, T)]^d$, for every right $S$-module $N$ and every $i\geq 0$.$C=\Hom_R(T, W)-T^d$ then 
$\T_C=\{M_S\mid \Hom_S(M, C)=0\}=\{M_S\mid M\otimes_ST\}$ is a torsion class in $\Modr S$ with corresponding torsion free class $\Cogen C\subseteq ^\perp C$. Let $(\T, \F)$ be the torsion pair in $\Modr R$ induced by $T_R$. There are equivalences
%\begin{enumerate}
%\item[(i)] 
\[\xymatrix{\T \ar@/_1pc/
[rrr]_{\Hom_R(T, -)} &&&\Cogen C\cap \E_{\perp}  \ar@/_1pc/[lll]_{-\otimes_S T}} 
\]

%\[\T\overset{\Hom_R(T, -)}\longrightarrow\Cogen C\cap \E_{\perp}\]
 %\item[(ii)] 
 \[\xymatrix{\F \ar@/_1pc/
[rrr]_{\Ext^1_R(T, -)} &&&\T_C\cap \E_{\perp}  \ar@/_1pc/[lll]_{\Tor_1^S(- T}} 
\]

%\end{enumerate}
\end{rem}
  \begin{prop}\label {P:case-n=1} Let $\clH$ be the heart of the $t$-structure induced by a good $1$-tilting module $T_R$ with endomorphism ring $S$ and let $\E$ be  as in Definition~\ref{D:E}. The following are equivalent:
\begin{enumerate}
\item $\clH$ is a Grothendieck category.
\item $\E$ is a hereditary torsion class in $\Modr S$.
\item For every $M\in \Modr S$, $\mathbb L G(M)\in \clH$.
\end{enumerate}
\end{prop}
\begin{proof} In view of Theorem~\ref{T:E-hereditary} it is enough to prove the implication (2) $\Rightarrow$ (3).

Let $M$ be a right $S$-module. Since $\Tor^S_i(-, T)=0$ for every $i\geq 2$, $\mathbb L G(M)\in \clH$ if and only if $\Tor^S_1(M, T)\in T^{\perp_0}$. Let $M_{\E}$ be the torsion submodule of $M$ with respect to the torsion pair $(\E, \E^{\perp_0})$. Then, $\Tor_1^S(M, T)\cong \Tor_1^S(M/M_{\E}, T)$, hence, w.l.o.g. we may assume that $M$ is $\E$-torsion free. 
Let now $M_C$ be the torsion submodule of $M$ with respect to the torsion class $\T_C$ (see Remark~\ref{R:C-torsion}). Then, $M/M_C\in \Cogen C\subseteq ^\perp C$, hence $\Tor^S_1(M/M_C, T)=0$, so that we can even assume that $M\in \T_C\cap \E^{\perp_0}$. As in the proof of (2) $\Rightarrow $ (3) in Theorem~\ref{T:E-hereditary}, there is an exact sequence $0\to M\to Y\to Y/M\to 0$ with $Y\in \E_{\perp}$ and $  Y/M\in \E$, then $\Tor_1^S(M, T)\cong \Tor_1^S(Y, T)$. Since $E\otimes_ST=0$ for every $E\in \E$ we have that $\E\subseteq \T_C$, hence $Y/M\in \T_C$. Thus $Y\in \E_{\perp}\cap \T_C$ and by Remark~\ref{R:C-torsion}  $ \Tor^S_1(Y, T)\in \F$ and thus, also $\Tor_1^S(M, T)\in \F$.
  \end{proof}

%

%%%%%%%%%%%%%%%%%%%
%%%%%%%%%%%%%%%%%%%
\section{Computing direct limits in the heart}~\label{s:dir-lim-heart}
In this section $\clH$ will always denote the heart of the $t$-structure induced by an $n$-tilting module $T_R$.

We apply the characterization proved by Theorem~\ref{T:E-hereditary} to show some properties of the category $\clH$.

  \begin{prop}\label{P:coproduct-closure} Assume that  $\clH$ is a Grothendieck category. Then the following hold true:
\begin{enumerate}
\item $\clH$ is closed under coproducts in $\D(R)$.
\item The classes $T^{\perp_i}$ are closed under direct sums in $\Modr R$, for every $i\geq 0$.
\end{enumerate}
\end{prop}
\begin{proof} (1) Let $(X_{\alpha}, \alpha\in \Lambda)$ be a family of objects in  $\clH$ and let $\Hom_{\clH}(T, X_{\alpha})=M_{\alpha}\in \Modr S$. By Theorem~\ref{T:E-hereditary}   we have $\mathbb L G(M_{\alpha})\cong X_{\alpha}$ and  $\mathbb L G(\bigoplus\limits_{\alpha\in \Lambda} M_{\alpha})\cong \coprod\limits_{\alpha\in \Lambda}\mathbb L G(M_{\alpha})$ belongs to $\clH$, hence the coproduct $\coprod\limits_{\alpha\in \Lambda}X_{\alpha}$ in $\D(R)$ belongs to $\clH$.

(2) It is clear that $T^{\perp_0}$ is closed under direct sums. By Proposition~\ref{P:coproducts} condition (1) implies that  $T^{\perp_1}$ is closed under direct sums.
We prove the statement by induction. Let $(N_{\alpha}, {\alpha\in \Lambda})$ be a family of $R$-modules in $T^{\perp_i}$ with $i>1$ and for every $\alpha$ consider a special $T^\perp$-preenvelope $0\to N_{\alpha}\to B_{\alpha}\to A_{\alpha}\to 0$ of $N_{\alpha}$; then $A_{\alpha}\in T^{\perp_{i-1}}$.  Consider the exact sequence $0\to \oplus N_{\alpha}\to \oplus B_{\alpha}\to \oplus A_{\alpha}\to 0$. $\oplus B_{\alpha}$ belongs to the tilting class, hence $\Ext^i_R(T, \oplus N_{\alpha})\cong \Ext^{i-1}_R(T, \oplus A_{\alpha})$ and the latter is zero by induction.
\end{proof}
\begin{prop}\label{P:direct-limits} Assume that $\clH$ is a Grothendieck category.
Consider a direct system $(X_{\alpha}; f_{\beta\alpha})_{\alpha\in \Lambda}$ of objects of $\clH$ and let  $(M_{\alpha}; g_{\beta\alpha})_{\alpha\in \Lambda}$ be the corresponding direct system of right $S$-modules obtained by applying the functor $\Hom_{\clH}(T,-)$.
Let $M=\varinjlim\limits_{\Modr S}M_{\alpha}$, then $\varinjlim\limits_{\clH} X_{\alpha}\cong \mathbb L G(M)$.

 In particular, for every $i\in \bbZ$, there are canonical isomorphisms:
 \[H^{-i}(\varinjlim\limits_{\clH} X_{\alpha})\cong\varinjlim\limits_{\Modr R}H^{-i}(X_{\alpha}).\]
 \end{prop}
 \begin{proof} By Theorem~\ref{T:E-hereditary} the restriction of the functor $\mathbb L G$ to $\Modr S$ is exact and left adjoint of the fully faithful functor $\Hom_{\clH}(T, -)$.
  Thus, $\mathbb L G (\varinjlim\limits_{\Modr S}M_{\alpha})\cong \varinjlim\limits_{\clH} \mathbb L G (M_{\alpha})$ and  for every $\alpha$ we have $\mathbb L G (M_{\alpha})\cong X_{\alpha}$. Hence the conclusion.

 In particular,  \[H^{-i}(\varinjlim\limits_{\clH} X_{\alpha})\cong \Tor^S_i(\varinjlim\limits_{\Modr S}M_{\alpha}, T)\cong \varinjlim\limits_{\Modr R} \Tor^S_i(M_{\alpha}, T)\cong\varinjlim\limits_{\Modr R}H^{-i}(X_{\alpha}) .\]
 \end{proof}

We show now that the last statement in Proposition~\ref{P:direct-limits} gives indeed a characterization of the Grothendieck condition of $\clH$.

\begin{thm}\label{T:homology-in-limits} The heart $\clH$ is a Grothendieck category if and only if for every direct system $(X_{\alpha}; f_{\beta\alpha})$ of objects of $\clH$
\[(\ast)\quad H^{-i}(\varinjlim\limits_{\clH} X_{\alpha})\cong\varinjlim\limits_{\Modr R}H^{-i}(X_{\alpha})\] for every $i\in \bbZ$.
\end{thm}
\begin{proof} Only the sufficiency needs to be proved.
 We follow the arguments as in the proof of \cite[Proposition 3.4]{PS15}.  Let $0\to \{X_{\alpha}\}\to \{Y_{\alpha}\}\to \{Z_{\alpha}\}\to 0$ be an exact sequence of direct systems of objects in $\clH$.
%we have to show  the exactness of the sequence  $0\to \varinjlim\limits_{\clH} X_{\alpha}\overset f \to \varinjlim\limits_{\clH} Y_{\alpha}\overset g \to \varinjlim\limits_{\clH} Z_{\alpha}\to 0$.
 Since the direct limit functor is right exact being a  left adjoint, there is and exact sequence 
  \[ \varinjlim\limits_{\clH} X_{\alpha}\overset f \to \varinjlim\limits_{\clH} Y_{\alpha}\overset g \to \varinjlim\limits_{\clH} Z_{\alpha}\to 0,\] 
  giving rise to short exact sequences:  $0\to \Img f\to \varinjlim\limits_{\clH} Y_{\alpha}\to \varinjlim\limits_{\clH} Z_{\alpha}\to 0$ and
$0\to \Ker f \to\varinjlim\limits_{\clH} X_{\alpha}\overset p\to \Img f\to 0$. Applying the cohomological functor $H$ to the triangles in $\D(R)$ corresponding to the above exact sequences and using the fact that direct limits are exact in $\Modr R$, we obtain a commutative diagram of $R$-modules:
 \[\xymatrix{\dots\ar[r]&\varinjlim\limits_{\Modr R} H^{-i-1}(Z_{\alpha})\ar[d]_{\cong}\ar[r]&\varinjlim\limits_{\Modr R} H^{-i}(X_{\alpha})\ar[r]\ar[d]_h&\varinjlim\limits_{\Modr R} H^{-i}(Y_{\alpha})\ar[d]^{\cong}\ar[r]&\dots\\
 \dots\ar[r]&H^{-i-1}(\varinjlim\limits_{\clH} (Z_{\alpha})\ar[r]&H^{-i}(\Img f)\ar[r]&H^{-i}(\varinjlim\limits_{\clH} Y_{\alpha})\ar[r]&\dots
 %\ar[d]_{\sim}H{i-1}(\varinjlim\limits_{\clH} Z_{\alpha})&H{i}(\Img f)\ar[r]&H{i}(\varinjlim\limits_{\clH} Y_{\alpha})&\dots
 }\]
thus $h$ is an isomorphism. Note that $h$ factors as \[\varinjlim\limits_{\Modr R} H^{-i}(X_{\alpha})\to H^{-i}(\varinjlim\limits_{\clH} (X_{\alpha})\overset{H^{-i}(p)}\to H^{-i}(\Img f)\] showing that  $H^{-i}(p)$ is an isomorphism for every $i\in \bbZ$, hence $p$ is an isomorphism and thus $\Ker f=0$.

\end{proof}

If $\clH$ is a Grothendieck category we show that some direct limits in $\clH$ can be computed in $\Ch (R)$.

 \begin{prop}\label{P:limits-in-Ch} Let
$(X_{\alpha}; f_{\beta\alpha})_{\alpha\in \Lambda}$ be a direct system in $\Ch (R)$
 % of complexes with terms in $T^\perp$ and concentrated in degrees $-n, \dots -1, 0$, 
such that  $X_{\alpha}\in \clH$ for every $\alpha\in \Lambda$.
If $\clH$ is a Grothendieck category then:  
\[\varinjlim\limits_{\Ch (R)} X_{\alpha}\cong \varinjlim\limits_{\clH} X_{\alpha}\] in $\D(R)$.
\end{prop}

\begin{proof} Consider the direct system  $(X_{\alpha}; q( f_{\beta\alpha}))_{\alpha\in \Lambda}$ in $\clH$ where $q$ is the canonical quotient functor $q\colon \Ch (R)\to \D(R)$. Let  $(M_{\alpha}; g_{\beta\alpha})_{\alpha\in \Lambda}$ be the direct system of right $S$-modules obtained by applying the functor $\Hom_{\clH}(T,-)$ to  $(X_{\alpha}; q( f_{\beta\alpha}))_{\alpha\in \Lambda}$. By Proposition~\ref{P:direct-limits}.
$\varinjlim\limits_{\clH} X_{\alpha}\cong \mathbb L G(M)$, where $M=\varinjlim\limits_{\Modr S}M_{\alpha}$.
By Lemma~\ref{L:counit}, there are projective resolutions $P_{\alpha}$ of $M_{\alpha}$ and a direct system $(P_{\alpha};\tilde{g}_{\beta\alpha})_{\alpha\in \Lambda}$ in $\Ch(S)$ such that $P_{\alpha}\otimes_ST\cong M_{\alpha}\Lotimes_ST$ and 
\[\varinjlim\limits_{\clH} X_{\alpha}\cong\mathbb L G(M)\cong (\varinjlim\limits_{\Ch (S)}P_{\alpha})\otimes_ST\cong \varinjlim\limits_{\Ch (R)}(P_{\alpha}\otimes_ST).\]
By Proposition~\ref{P:left-adjoint} and its proof, the functor  $\Hom_{\clH}(T,-)$ is isomorphic to $\RHom_R(T, -)$ and it is fully faithful. Thus, the counit morphism \[-\Lotimes_ST\circ  \Hom_{\clH}(T,-)\] is invertible (see Fact~\ref{F:good-tilting}) showing that  $X_{\alpha}\cong M_{\alpha}\Lotimes_ST\cong P_{\alpha}\otimes_ST$ in $\clH$, for every $\alpha\in \Lambda$. Let $\phi_{\alpha}\colon X_{\alpha}\to P_{\alpha}\otimes_ST$ be an isomorphism and let $\psi_{\alpha}$ be a chain map in $\Ch(R)$ such that $q(\psi_{\alpha})=\phi_{\alpha}$. The map
\[\psi=\varinjlim\limits_{\alpha}\psi_{\alpha}=\colon\varinjlim\limits_{\Ch (R)} X_{\alpha}\to\varinjlim\limits_{\Ch (R)}(P_{\alpha}\otimes_ST)\] is a chain map in $\Ch (R)$.
Consider the morphisms:
\[ \varinjlim\limits_{\Modr R}H^{-i}(X_{\alpha})\overset{g}\to H^{-i}(\varinjlim\limits_{\Ch (R)}X_{\alpha})\overset{H^{-i}(\psi)}\to H^{-i}(\varinjlim\limits_{\clH}X_{\alpha}),\]
where $g$ is a canonical isomorphism by the exactness of the direct limit in $\Modr R$ and the composition $H^{-i}(\psi)\circ g$ is an isomorphism by Proposition~\ref{P:direct-limits}. Hence $H^{-i}(\psi)$ is an isomorphism for every $i\in \bbZ$ implying that $\psi$ is an isomorphism in $\D(R)$. 
\end{proof}
\section{The pure projectivity}\label{S:7}

 In this section we translate the Grothendieck condition on the category $\clH$ in terms of properties of subcategories of $\Modr R$ in order to be able to pin down  conditions on the tilting module $T_R$ itself.

First we prove a result which is a consequence of Proposition~\ref{P:limits-in-Ch}, where $\tr_T$ denotes the trace in the module $T$.
 \begin{prop}\label{P:trace-in-limits} Assume that $\clH$ is a Grothendieck category. The following hold true:\begin{enumerate}
 \item
If $(N_{\alpha}; f_{\beta\alpha})_{\alpha\in \Lambda}$ is a direct system of $R$-modules, then
 \[ \tr_T(\varinjlim_{i\in I}N_{\alpha})\cong \varinjlim_{i\in I}\tr_T(N_{\alpha}).\]
 
 In particular, the torsion free class $T^{\perp_0}$ is closed under direct limits, hence it is a definable class.
 \item For every $i\geq 1$, the classes $T^{\perp_i}$ are closed under direct limits.
 \end{enumerate}
 \end{prop}
 \begin{proof} 
 (1) For each module $N_{\alpha}$ choose functorially a complex $X_{\alpha}\in \clH$ as constructed in Proposition~\ref{P:N-in-orthogonal}~(1). Then $\Ker d_{X_{\alpha}}^{-1}=N_{\alpha}$ and  $\Img d_{X_{\alpha}}^{-2}=\tr_T(N_{\alpha})$. By functoriality we obtain a direct system $(X_{\alpha}; \tilde f_{\beta\alpha})_{\alpha\in \Lambda}$ in $\Ch (R)$ and also a direct system $(X_{\alpha}; q(\tilde f_{\beta\alpha}))_{\alpha\in \Lambda}$ in $\clH$ where $q$ is the canonical quotient functor $q\colon \Ch (R)\to \D(R)$. Let $X=\varinjlim\limits_{\Ch (R)} X_{\alpha}$. 
 By Proposition~\ref{P:limits-in-Ch}, $\varinjlim\limits_{\clH} X_{\alpha}\cong X.$ 
 
$X$ is a complex with terms in $T^\perp$, since an $n$-tilting class is closed under direct limits (see\cite{BS07}).
We have: $\Ker d_X^{-1}= \varinjlim\limits_{\Modr R} N_{\alpha}$; $\Img d_X^{-2}=\varinjlim\limits_{\Modr R} \tr_T(N_{\alpha})$ and by Lemma~\ref{L:description}~(2), $\Img d_X^{-2}=\tr_T(\Ker d_X^{-1})=\tr_T( \varinjlim\limits_{\Modr R} N_{\alpha})$, since $X\in \clH$. Hence the conclusion.

The last statement follows immediately by $(\ast)$. 
%Note that it could be shown directly by choosing functorially for each $N_{\alpha}$ in $T^{\perp_0}$ a complex $X_{\alpha}\in \clH$ as constructed in Proposition~\ref{P:N-in-orthogonal}~(2). Then $H^{-i}(X_{\alpha})=0$ for every $i\geq 2$ and $H^{-1}(X_{\alpha})\cong M_{\alpha}$. Let again $X=\varinjlim\limits_{\clH}X_{\alpha}$, then by Proposition~\ref{P:direct-limits}, $H^{-i}(X)=0$ for every $i\geq 2$ and $H^{-1}(X)\cong \varinjlim\limits_{\alpha}H^{-1}(X_{\alpha})\cong \varinjlim\limits_{\alpha}N_{\alpha}$. By Lemma~\ref{L:description}~(3)~(c)  $H^{-1}(X)\in T^{\perp_0}$.

(2) We first prove that the class $T^{\perp_1}$ is closed under direct limits.

Let $(N_{\alpha}; f_{\beta\alpha})_{\alpha\in \Lambda}$ be a direct system of $R$-modules in $T^\perp_1$. For each module $N_{\alpha}$ choose functorially a complex $X_{\alpha}\in \clH$ as constructed in Proposition~\ref{P:N-in-orthogonal}~(2) so that $N_{\alpha}=\Ker d^{-2}_{X_{\alpha}}$. Arguing as in part (1) we get a direct system $\{X_{\alpha}\}_{\alpha\in \Lambda}$ both in $\Ch (R)$ and in $\clH$ whose direct limit in $\clH$ is isomorphic to  $X=\varinjlim\limits_{\Ch (R)} X_{\alpha}$, by Proposition~\ref{P:limits-in-Ch}. Now, $\Ker d^{-2}_X\cong \varinjlim\limits_{\Modr R} \Ker d^{-2}_{X_{\alpha}}$ and by Lemma~\ref{L:description}~(2), the latter belongs to $T^\perp_1$.

By induction we get that $T^{\perp_i}$ is closed under direct limits for every $i\geq 1$. In fact, let $(N_{\alpha}; f_{\beta\alpha})_{\alpha\in \Lambda}$ is a direct system of $R$-modules in $T^{\perp_i}$, with $i>1$ and choose functorially  special $T^\perp$-preenvelopes of $N_{\alpha}$ of the form $0\to N_{\alpha}\to B_{\alpha}\to A_{\alpha}\to 0$, with $B_{\alpha}\in T^\perp$ and $A_{\alpha}\in \A$. Then $A_{\alpha}\in T^{\perp_{i-1}}$. We obtain a short exact sequence 
\[0\to \varinjlim\limits_{\alpha} N_{\alpha}\to \varinjlim\limits_{\alpha} B_{\alpha}\to \varinjlim\limits_{\alpha} A_{\alpha}\to 0.\]

Since $T^\perp$ is closed under direct limits, $\varinjlim\limits_{\alpha} N_{\alpha}\in T^{\perp_{i}}$ if and only if $\varinjlim\limits_{\alpha} A_{\alpha}\in T^{\perp_{i-1}}$. Thus, the conclusion follows by induction.

%By \cite[Theorem 6.1]{AST15} the classes $T^{\perp_i}$ are definable.
\end{proof}
%
%Recall that a module is pure projective
% if and only if it has the projective property with respect to pure exact sequences.
%
%
%By Warfield~\cite{War1} a module $M$ is pure projective if and only if every pure exact sequence $0\to A \to B \to M\to 0$ splits or, equivalently, if and only if it is a direct summand of a direct sum of finitely presented modules.
%

We show now that if the heart $\clH$ is a Grothendieck category, then the $n$-tilting module $T$ must be pure projective.

\begin{prop}\label{P:pure-projective} Let $\clH$ be the heart of the $t$-structure induced by an $n$-tilting module $T$. If $\clH$ is a Grothendieck category, then $T$ is a pure projective module.
\end{prop}
\begin{proof} Write $T$ as a direct limit of a direct system $\{A_i\colon f_{ji}\}_{i\leq j\in I}$ of finitely presented modules. By Proposition~\ref{P:trace-in-limits} $T\cong\varinjlim\limits_{i} \tr_T(A_i)$ and for every $i\in I$ we have a functorial presentation of $\tr_T(A_i)$ given by
\[0\to K_i\to T^{(\Hom(T, A_i))}\to \tr_T(A_i)\to 0,\]
where $K_i\in T^{\perp_1}$. By the functoriality of the presentation we get direct systems $\{K_i\}_{i\in I}$, $\{T^{(\Hom(T, A_i))}\}_{i\in I}$ and $\{\tr_T(A_i)\}_{i\in I}$ giving rise to the following commutative diagram:
 \[\xymatrix{&0\ar[d]&0\ar[d]&0\ar[d]\\
 0\ar[r]&D_1\ar[d]\ar[r]&D\ar[r]\ar[d]&D_2\ar[d]\ar[r]&0\\
 0\ar[r]&\bigoplus\limits_{i\in I}K_i\ar[d]\ar[r]&\bigoplus\limits_{i\in I}T^{(\Hom(T, A_i))}\ar[r]^{\sigma}\ar[d]_{\phi}&\bigoplus\limits_{i\in I}\tr_T(A_i)\ar[d]^{\pi}\ar[r]&0\\
0\ar[r]&\varinjlim\limits_{i\in I} K_i\ar[r]\ar[d]&\varinjlim\limits_{i\in I}T^{(\Hom(T, A_i))} \ar@/_/[r]_{\rho}\ar[d]&\varinjlim\limits_{i\in I}\tr_T(A_i)\cong T\ar@/_/[l]^{\alpha}\ar@{.>}[ul]_{\beta}\ar[d]\ar[r]&0\\\
&0&0&0 }.\]
By Proposition~\ref{P:trace-in-limits}~(2), $\varinjlim\limits_{i\in I} K_i$ is in $T^{\perp_1}$, hence the last row splits that is, there is a morphism $\alpha\colon T\to \varinjlim\limits_{i\in I}T^{(\Hom(T, A_i))} $ such that $\rho\circ\alpha=1_T$. The second column is a pure exact sequence, hence $D\in T^\perp$ since the tilting class $T^\perp$ is definable (by \cite{BS07}). This implies that the morphism $\alpha$ can be lifted to a morphism $\beta$ such that $\phi\circ\beta=\alpha$.
 Now we infer that $\pi\circ \sigma\circ\beta=\rho\circ\phi\circ\beta=\rho\circ \alpha=1_T$ showing that the morphism $\sigma\circ\beta$ gives a splitting map for the third column. We then conclude that $T$ is isomorphic to a direct summand of $\bigoplus\limits_{i\in I}\tr_T(A_i)$. We also have a commutative diagram
  \[\xymatrix{
 0\ar[r]&D_1\ar[d]\ar[r]&\bigoplus\limits_{i\in I}\tr_T(A_i)\ar[r]\ar[d]&\varinjlim\limits_{i\in I}\tr_T(A_i)\cong T\ar[d]_{\cong}\ar[r]&0\\
 0\ar[r]&D'\ar[r]&\bigoplus\limits_{i\in I}A_i\ar[r]&\varinjlim\limits_{i\in I}A_i\cong T\ar[r]&0\\
 }\]
where the first row splits, showing that also the second row splits. This proves that $T$ is pure projective.
 
\end{proof}
\begin{rem}\label{R:pure-projective} Note that the proof of Proposition~\ref{P:pure-projective} shows that if $T$ is an $n$-tilting module, the following two conditions: (1) the  functor $\tr_T$ commutes with direct limit and (2) $T^{\perp_1}$ is closed under direct limits, are sufficient to conclude that $T$ is pure projective.
\end{rem}

We use in argument suggested by Ivo Herzog to prove the converse of the preceding proposition.
\begin{prop}\label{P:pure-projective-suff} Let $T$ be a pure projective $n$-tilting module. Then the heart $\clH$ of the $t$-structure induced by $T$ is a Grothendieck category.
\end{prop}
\begin{proof} Consider the functor category $\A=((\modr R)^{\text op}, \A b)$ consisting of the contravariant additive functors form the category of finitely presented right $R$-modules to the category of abelian groups.
It is well known that the Yoneda functor $H\colon \Modr R\to \A;\quad M\mapsto \Hom_R(-, M)$ yields a left exact full embedding  and that $\Hom_R(-, M)$ is a projective object of $\A$ provided that $M$ is a pure projective $R$-module.
Thus, by assumption $H_T=\Hom_R(-, T)$ is a projective object  of $\A$ and the class 
\[\C=
\{G\in \A\mid \Hom_{\A}(H_T, G)=0\}\] is a torsion torsion free class, so that we can form the quotient category $\A/\C$. By \cite[Ch. III] {G62} $\A/\C$ is a Grothendieck category and the quotient functor $q\colon \A\to \A/\C$ is exact. The group of morphisms $\Hom_{\A/\C}(q(G), q(F)$ between two objects in $\A/\C$ is defined by $\varinjlim\Hom_{\A}(G', F/F')$ where $G'$ and $F'$ vary among the subobjects of $G$ and $F$ such that $G/G', F'\in \C$.
 Thus, by the definition of $\C$ and by the projectivity of $H_T$ we infer that $\Hom_{\A/\C}(q(H_T), q(F))\cong \Hom_{\A}(H_T, F)$ which yields that $q(H_T)$ is a projective object of $\A/\C$. Moreover, from the definition of $\C$ it is clear that $q(H_T)$ is a generator for $\A/\C$. Thus $\Add\,q(H_T)$ is the class of projective objects of the Grothendieck category $\A/\C$ and the composition of functors:
 \[\Modr R\overset{H}\longrightarrow\A\overset{q}\longrightarrow \A/\C\]
 induces an equivalence between $\Add T$ and  $\Add\, q(H_T)$.
 
 By Proposition~\ref{P:resol-dimension}, the full subcategory of projective objects of $\clH$ is equivalent to $\Add T$ so we have an equivalence between the full subcategories of projective objects of $\clH$ and of $\A/\C$. It is well known that the equivalence extends to the entire categories (see e.g. \cite[IV]{ARS}), thus we conclude that $\clH$ is a Grothendieck category.
 \end{proof}

  Combining Proposition~\ref{P:pure-projective} with Proposition~\ref{P:pure-projective-suff} we obtain the main result of this section.
  \begin{thm}\label{T:characterization} Let $\clH$ be the heart of the $t$-structure induced by an $n$-tilting module $T$. 
$\clH$ is a Grothendieck category if and only the tilting module $T$ is  pure projective.
\end{thm}

We illustrate now some properties of the trace functor corresponding to an $n$-tilting module.
 If $n=1$, $\tr_T$ is the torsion radical of the tilting torsion class $T^\perp=\Gen T$.
 
 For $n>1$ we have:
 \begin{lem}\label{L:gen-T} Let $T$ be an $n$-tilting $R$-module  and consider a special $T^{\perp}$-preenvelope of $R$ of the form \[
(\ast)\quad 0\to R\overset{\varepsilon}\to T\to T/R\to 0,\]
and let $\varepsilon (1)=w$.
Then, for every module $N$ there are two exact sequences:
 \[(1)\quad 0\to \Hom_R(T/R, N)\to \Hom_R(T, N)\to tr_T(N)\to 0,\]
  \[(2)\quad 0\to N/tr_T(N)\to \Ext^1_R(T/R, N)\to \Ext^1_R(T, N)\to 0.\]
  In particular the following hold:
  \begin{enumerate}
  \item[(i)]
For every module $N$, $x\in \tr_T(N)$ if and only if there is a morphism $g\colon T\to N$ such that $g(w)=x$.
  In other words $\tr_T $ is isomorphic to the matrix functor $\Hom_R(T,-)(w)$.
  \item[(ii)] The trace $\tr_T$  commutes with direct limits if and only if there is a finitely presented module $A$ and an element $a\in \tr_T(A)$ such that $\tr_T$ is isomorphic to the finite matrix functor $\Hom_R(A, -)(a).$
  \end{enumerate}
 \end{lem}

 \begin{proof} Possibly passing to an equivalent  tilting module it is easy to see that there exists a special $T^{\perp}$-preenvelope of $R$ as in the statement.
 
 From $(\ast)$ we obtain the exact sequence:
 
  \[ (a)\quad 0\to \Hom_R(T/R, N)\to \Hom_R(T, N)\to \Hom_R(R, N)\to\]
  \[\to \Ext^1_R(T/R, N)\to \Ext^1_R(T, N)\to 0.\]

  Identifying $\Hom_R(R, N)$ with $N$, it is obvious that the image of the map $\Hom_R(T, N)\to N$ is contained in the trace of $T$ in $N$. To prove the other inclusion pick $x\in tr_T(N)$ and let $f\colon R\to tr_T(N)$ be a morphism satisfying $f(1)=x$ and let $T^{(\gamma)}\overset{\phi}\to tr_T(N)$ be an epimorphism. Consider the diagram:
 \[\xymatrix{
0\ar[r]&R\ar[r]^{\varepsilon}\ar[dr]^>f\ar@{-->}[d]_h&T\ar@{-->}[dl]_>>{\ell}\\
&T^{(\gamma)}\ar[r]_{\phi}&tr_T(N)\ar[r]&0}
\]
where the dotted line $h\colon R\to T^{(\gamma)}$ satisfies $\phi\circ h= f$ and the dotted line $\ell\colon T\to T^{(\gamma)}$ satisfying  $\ell\circ\varepsilon=h$ exists by the preenvelope property.  Then the morphism $g=\phi\circ \ell\colon T\to tr_T(N)$ satisfies $g\circ\varepsilon =f$, hence $g(w)=x$ and sequence (1) is established.

 Sequence (2) follows from $(a)$  and (1).
 
 (i) follows immediately by (1).
 
 (ii) Assume that $\tr_T$ commutes with direct limits and write $T$ as a direct limit of a direct system $\{A_i\}_{i\in I}$ of finitely presented modules. By assumption $T\cong \varinjlim\limits_{i\in I}\tr_T(A_i)$. Let  $\mu_i\colon A_i\to T$ be the canonical morphisms. There is an index $j\in I$ and an element $a_j\in \tr_T(A_{j})$ such that $\mu_{j}(a_j)=w$. Let $N$ be an $R$-module and let $x\in \tr_T(N)$. By  (i)  and the above remarks, there is $f\colon A_{j}\to N$ such that $f(a_j)=x$.
 Hence, $\Hom_R(T, N)(w)\leq \Hom_R(A_j, N)(a_j)$. 
 
 On the other hand $ \Hom_R(A_j, N)(a_j)\leq  \Hom_R(\tr_T(A_j), N)(a_j)$ and the latter is contained in  $\tr_T(N)$, since  $tr_T(A_j)$ is generated   by $T$. Hence $\tr_T$ is isomorphic to $\Hom_R(A_j, -)(a_j)$.
 
The converse follows immediately by recalling that, for every finitely presented module $A$, the functor  $\Hom_R(A, -)$ commutes with direct limits. \end{proof}

 \begin{rem}\label{R:condition-on-trace} 
 The condition on $\tr_T$ to commute with direct limits doesn't seem to imply the pure projectivity of $T$. In fact, in Proposition~\ref{P:pure-projective} to prove that $T$ is pure projective we used also that $T^{\perp_1}$ is closed under direct limits.
 \end{rem}
 
 We illustrate now some features of a pure projective $n$-tilting module.
  
\begin{prop}\label{P:syzygies-of-T} Let $T$ be a pure projective $n$-tilting module. Then every $i^{th}$-syzygy of $T$ is pure projective.
\end{prop}
\begin{proof} First of all we show that $T$ may be assumed to be countably presented. By assumption $T$ is a direct summand of a direct sum $\bigoplus\limits_{i\in I}E_{i}$ of finitely presented modules $E_i$, hence in particular countably generated.
 By Kaplansky's Theorem ~\cite[Theorem 1]{Ka}, $T$ is a direct sum of countably generated submodules.  Thus $T=\bigoplus\limits_{\alpha\in \Lambda}X_{\alpha}$ where for every $\alpha$, $X_{\alpha}$ is a countably generated, hence also countably presented,  direct summand of $\bigoplus\limits_{i\in I}E_{i}$. Let $A$ be a countably presented module in the left component $\A$ of the cotorsion pair $(\A,  
T^\perp)$ generated by $T$ and consider a special $T^\perp$-preenvelope of $A$ \[0\to A\overset{\epsilon}\to B\to A_1\to 0.\] W.l.o.g we may assume that $B=T^{(\gamma)}$ for some cardinal $\gamma$
 so that $\varepsilon (A)$ is contained in a summand $U_0= (\bigoplus\limits_{\beta\in F}X_{\beta})^{(\nu)}$ of $T^{(\gamma)}$ where $F$ is a countable subset of $\Lambda$ and $\nu$ is a countable subset of $\gamma$. Thus also \[0\to A\to U_0\to U_0/A\to 0\]
  is a special preenvelope of $A$ and $U_0/A$ is a countably presented module in $\A$.
  Starting with $R\in \A$, the above arguments show that we can construct an iteration of $T^\perp$ preenvelopes of $R$  of the form\
  \[0\to R\to U_0\to U_1\to\dots\to U_n\to 0,\]
  with $U_i$ countably presented modules in $\Add T$, hence by \cite[Ch 13]{GT12} we obtain that $U_0\oplus U_1\oplus\dots\oplus U_n$ is an $n$-tilting module equivalent to $T$.
  
  Secondly, we observe that all the syzygies of $T$ are countably presented. Indeed this follows by recalling that every module in $\A$ is $R$ Mittag-Leffler (see for instance \cite[Theorem 9.5]{AHH}), and applying \cite[Proposition 3.8]{BH09} to the syzygies of $T$. Since countably generated modules are pure projective if and only if they are Mittag-Leffler (\cite{RG}), we are lead to show that every syzygy $\Omega_i(T)$ ($i\geq 1$) of $T$ is a Mittag-Leffler module, provided that $T$ is pure projective. This is equivalent to show that the canonical morphism $\rho_i\colon\Tor_i^R(T, \prod\limits_{i\in I} Q_i)\to \prod\limits_{i\in I}\Tor^R_i(T, Q_i)$ is a monomorphism, for every $i\geq 1$. (see e.g.\cite[Proposition 1.10]{AHH}).
   By dimension shifting it is enough to prove that $\rho_1$ is a monomorphism. By assumption $T$ is a summand of a direct sum $\bigoplus\limits_{n\in \bbN}E_{n}$ of finitely presented models $E_n$. So 
   \[\Tor_i^R(T, \prod\limits_{i\in I} Q_i)\underset{\oplus}\leq \Tor_i^R(\bigoplus\limits_{n\in \bbN}E_{n}, \prod\limits_{i\in I} Q_i)\cong \bigoplus\limits_{n\in \bbN}\left( \prod\limits_{i\in I} \Tor^R_1(E_n, Q_i)\right)
   \] and the latter can be embedded in $\prod\limits_{i\in I} (\bigoplus\limits_{n\in \bbN}\Tor^R_1(E_n, Q_i)$. Thus, by the naturality of $\rho_1$ we conclude that $\rho_1$ is a monomorphism.\end{proof}

We can  characterize the pure projectivity of an $n$-tilting module in terms of properties of its $\Ext$-orthogonal classes. First we prove the following lemma.
 \begin{lem}\label{L:syzygies} Let $M$ be a pure projective module. The following hold true:
 \begin{enumerate}
 \item  $M^{\perp_0}$ is a definable class and $M^{\perp_1}$ is closed under direct sums.
 \item If  all the sygyzies $\Omega_{j}(M)$ of $M$ are pure projective, then the classes $M^{\perp_i}$ are definable for every $i\geq 0$. 
 \end{enumerate}
 \end{lem}
 \begin{proof} (1) The pure projectivity of $M$ yields  easily that $M^{\perp_0}$ is closed under direct limits and thus $M^{\perp_0}$ is a definable class.
 
Let $X_{\alpha}$ be a family of modules in $M^{\perp_1}$. From the pure exact sequence \[(a)\qquad 0\to \oplus_{\alpha}X_{\alpha}\to \prod_{\alpha}X_{\alpha}\overset{\pi}\to  \dfrac{\prod_{\alpha}X_{\alpha}}{\oplus_{\alpha}X_{\alpha}}\to 0\]  and from the pure projectivity of $M$, we obtain that $\Hom_R(M, \pi)$ is surjective  and thus $\oplus_{\alpha}X_{\alpha}\in M^{\perp_1}$. 

(2) We prove the statement by induction. The case $i=0$ holds by (1).
Let $i=1$.  We first show that $M^{\perp_1}$ is closed under pure submodules. Let $0\to Y\to X\overset{\pi}\to X/Y\to 0$ be a pure exact sequence with $X\in M^{\perp_1}$. 
Then $\Hom_R(M, \pi)$ is surjective, hence the exact sequence $0\to \Ext^1_R(M, Y)\to \Ext^1_R(M, X)=0$, shows that 
$Y\in M^{\perp_1}$.

 Let $X_{\alpha}$ be a direct system of modules in $M^{\perp_1}$ and consider a pure exact sequence
 \[0\to K\to  \oplus_{\alpha}X_{\alpha}\to\varinjlim_{\alpha}X_{\alpha}\to 0.
 \]
By (1) $\oplus_{\alpha}X_{\alpha}\in M^{\perp_1}$ and we have an exact sequence
\[0\to \Ext^1_R(M, \varinjlim_{\alpha}X_{\alpha})\to \Ext^2_R(M, K)\overset{f}\to \Ext^2_R(M,\oplus_{\alpha}X_{\alpha}).\]
By dimension shifting we have a canonical isomorphism $\Ext^{2}_R(M, -)\cong \Ext^{1}_R(\Omega_{1}(M), -)$. Using the pure projectivity of $\Omega_{1}(M)$ we have a monomorphism $\Ext^1_R(\Omega_{1}(M), K)\to \Ext^1_R(\Omega_{1}(M),  \oplus_{\alpha}X_{\alpha})$ and by the naturally of the isomorphism we conclude that $f$ is a monomorphism, too. 

Hence $\Ext^1_R(M, \varinjlim_{\alpha}X_{\alpha})=0.$

To conclude the proof it is enough to note that $\Ext^{i+1}_R(M, -)\cong \Ext^{1}_R(\Omega_{i}(M), -)$ for every $i\geq 1$ and apply the previous arguments.
 \end{proof}

   \begin{prop}\label{P:definable} Let $T$ be an $n$-tilting module. The following are equivalent
   \begin{enumerate}
   \item $T$ is pure projective.
    \item $\tr_T$ commutes with direct limits and the classes $T^{\perp_i}$ are definable for every $0\leq i$.
   \item $\tr_T$ commutes with direct limits and the class $T^{\perp_1}$ is closed under direct limits.
   \end{enumerate}
     \end{prop}
   \begin{proof}
 (1) $\Rightarrow$ (2) The first statement follows by the definition of pure projectivity and by the canonical presentation of a direct limit by means of a pure exact sequence. In particular, $T^{\perp_0}$ is closed under direct limits and thus it is definable.
  
  For the closure under direct limits of the classes $T^{\perp_i}$, $i\geq 1$ we could invoke Theorem~\ref{T:characterization}, and Proposition~\ref{P:trace-in-limits} and then apply \cite[Theorem 6.1]{AST15}. Alternatively, we can use Proposition~\ref{P:syzygies-of-T} and Lemma~\ref{L:syzygies}.
  
   (2) $\Rightarrow$ (3) Obvious.
   
    (3) $\Rightarrow$ (1) Follows by the proof of Proposition~\ref{P:pure-projective}. See also Remark~\ref{R:pure-projective}.   \end{proof}

\end{document}